\documentclass[12pt,a4paper,oneside]{article}
\usepackage{geometry}
\usepackage{graphicx}
\usepackage{subfigure}
\usepackage{float}
\usepackage{color}
\usepackage{verbatim}
\usepackage{amssymb,amsmath}

\newcommand{\Rmnum}[1]{\expandafter\@slowromancap\romannumeral #1@}
\makeatletter
\def\Suma{\sum^{J/2-1}_{l=-J/2}}
\def\Sumb{\sum^{J-1}_{j=0}}
\newtheorem{theorem}{Theorem}[section]

\newtheorem{lemma} {Lemma}[section]

\newtheorem{remark}{Remark}[section]

\newtheorem{definition} {{Definition}}[section]

\newenvironment{pf}{{\noindent \it \bf Proof. }}{{\hfill$\Box$}\\}
\newenvironment{proof}{{\noindent \it \bf Proof. }}{{\hfill$\Box$}\\}
\makeatother
\newcommand{\enabstractname}{Abstract}

\begin{document}

\title{On a Fractional Schr\"odinger equation in the presence of Harmonic potential}

\author{Zhiyan Ding\footnote{Department of Mathematics, University of Wisconsin, Madison (zding49@wisc.edu)},\ \ \  Hichem Hajaiej\footnote{California State University, Los angeles, 5151 University Drive (hhajaie@calstatela.edu)}
}
\maketitle
\begin{abstract}
In this paper, we establish the existence of ground state solutions for a fractional Schr\"odinger equation in the presence of a harmonic trapping potential. We also address the orbital stability of standing waves. Additionally, we provide interesting numerical results about the dynamics and compare them with other types of Schr\"odinger equations \cite{R3,R4}. Our results explain the effect of each term of the Schr\"odinger equation : The fractional power, the power of the nonlinearity and the harmonic potential.
\smallskip

\smallskip

\end{abstract}
\maketitle

\section{Fractional Nonlinear Schr\"odinger equation with harmonic potential}

In this paper, we examine the following Schr\"odinger equation:
\begin{equation}\label{E1.1}
\left\{\begin{aligned}
  &i \psi_{t}=(-\Delta)^{s}\psi+|x|^{2}\psi-|\psi|^{2\sigma}\psi\;\;\mbox{in}\;\mathbb{R}^{N}\times [0,\infty),\\
  &\psi(0,x)=\psi_{0}(x)\in H_{s}(\mathbb{R}^{N}),
\end{aligned}
\right.
\end{equation}
where $0<s<1$, $\sigma>0$, $N\geq1$ and $\psi: \mathbb{R}^{N}\times [0,\infty)\longrightarrow \mathbb{C}$ is the wave function with initial condition $\psi_0(x)$ belongs to the following Sobolev space:
\begin{equation*}\label{E2.3}
H_{s}(\mathbb{R}^{N}):=\left\{u\in L^{2}(\mathbb{R}^{N}):\int_{\mathbb{R}^{N}}\int_{\mathbb{R}^{N}}\frac{|u(x)-u(y)|^{2}}{|x-y|^{N+2s}}dydx<\infty\right\}.
\end{equation*}
with 
\begin{equation*}\label{E2.31}
\|u\|_{H_{s}(\mathbb{R}^{N})}=\sqrt{\int_{\mathbb{R}^{N}}\int_{\mathbb{R}^{N}}\frac{|u(x)-u(y)|^{2}}{|x-y|^{N+2s}}dydx+\int_{\mathbb{R}^N}|u(x)|^2dx}
\end{equation*}

The fractional Laplacian $(-\Delta)^{s}$ is defined via a pseudo-differential operator
\begin{equation}\label{E1.2}
(-\Delta)^{s}u(x)=\mathcal{F}^{-1}[|\xi|^{2s}\mathcal[u]],\ s>0.
\end{equation}

For the Cauchy problem \eqref{E1.1}, we have two important conserved quantities: The mass of the wave function:
\begin{equation}\label{E1.3}
M(t)=||\psi(\cdot,t)||^2:=\int_{\mathbb{R}^N}|\psi(x,t)|^2dx\equiv M(0)
\end{equation}
and the total energy:
\begin{equation}\label{E1.31}
E(t)=\int_{\mathbb{R}^N}\left[\text{Re}\left(\psi^*(x,t)(-\Delta)^{s}\psi\right)+|x|^2|\psi|^2-\frac{1}{\sigma+1}|\psi(x,t)|^{2(\sigma+1)}\right]dx\equiv E(0).
\end{equation}

In  recent years,  a great attention has been focused on the study of problems involving the fractional Laplacian, which naturally appears in obstacle problems, phase transition, conservation laws, financial market. Nonlinear fractional Schr\"odinger equations have been proposed by Laskin \cite{las, lask} in order to expand the Feynman path integral, from the Brownian like to the L\'evy like quantum mechanical paths. The stationary solutions of fractional nonlinear Schr\"odinger equations have also been intensively studied due to their huge importance in nonlinear optics and quantum mechanics \cite{las, lask, HOW, OB}. The most interesting solutions have the special form:
\begin{equation}\label{E1.5}
\psi(x,t)=e^{-i\lambda t}u(x),\;\;\lambda\in \mathbb{R},\;\;u(x)\in \mathbb{C}.
\end{equation}
They are called the standing waves. These solutions reduce \eqref{E1.1} to a semilinear elliptic equation. In fact, after plugging \eqref{E1.5} into \eqref{E1.1}, we need to solve the following equation
\begin{equation}\label{E1.6}
 (-\Delta)^{s}u(x)+|x|^{2}u(x)-|u(x)|^{2\sigma}u(x)=\lambda u(x)\;\;\mbox{in}\;\mathbb{R}^{N}\times [0,\infty).
\end{equation}

The case $s=1$ has been intensively studied by many authors (See \cite{Zhao}). There also exist a considerable amount of results concerning the standing waves of fractional Nonlinear Schr\"odinger equations without the harmonic potential, we refer the readers to \cite{haj2, chan, cha, fib, haj, haj1, ros, Zh} and the references therein.

In this paper, we mainly focus on the solutions to \eqref{E1.6}. To the best of our knowledge, our results are new and will open the way to solve other class of fractional Schr\"odinger equations. This paper has two main parts: In the first part, we address the existence of standing waves through a particular variational form, whose solutions are called ground state solutions. We prove the existence of ground state solutions (Theorem \ref{T1}), and show some qualitative properties like monotonicity and radiality (Lemma \ref{R3.2}). We also proved that the ground state solutions are orbitally stable (Def \ref{D4.1}, Theorem \ref{T2}) if we have the uniqueness of the solutions for the Cauchy problem \eqref{E1.1} (Theorem \ref{T4.1}). We have also addressed the critical case $\sigma=\frac{2s}{N}$, which is consistent with the case $s=1$ in \cite{R3}. The second part of this article deals with the numerical method to solve \eqref{E1.1} and to establish the existence of ground state solutions as well as to establish the optimality of our conditions. In this part, we were not only able to show the existence of ground state solutions for $0<\sigma<\frac{2s}{N}$ but we also gave a  constrained variational problem (\eqref{E6.15}-\eqref{E6.17}), which was crucial to find the standing waves for the subcritical $\frac{2s}{N}\leq\sigma< \frac{2s}{N-2s}$. The numerical results provided a good explanation of the effect of $s$ on the ground state solution. To reach this goal, we showed the ground state solution is continuous and decreasing with respect to $s$ in $L^2$ and $L^\infty$ norm (Figure \ref{F2}), which is a similar phenomenon to \cite{R4}. Besides, like Gross–Pitaevskii Equation \cite{R3}, we examined the convergence property of $\lambda_c$. It turned out this convergence property also holds true in our case. Second, we checked the stability of ground state solutions for different $s$. If we add a small perturbation to the initial condition, for different $s$, the absolute value of the solution will always have periodic behavior, which shows the orbital stability (Figure \ref{F4}). Furthermore, surprisingly, when $s$ becomes smaller, the stability is worse, which means the oscillation amplitude in the periodic phenomenon becomes larger (Figure \ref{F5},\ref{F6}). We then address the case where the harmonic potential is not radial, and we obtained non radial symmetrical ground state solution (Figure \ref{F10.1},\ref{F10.2}). Finally, we provided interesting numerical results for the time dynamics of FNLS. 

The main difficulty of constructing ground state solutions comes from the lack of compactness of the Sobolev embeddings for the unbounded domain $\mathbb{R}^N$. However, by defining an appropriate function space, in which the norm of the potential is involved, we "recuperate" the compactness (see Lemma \ref{L3.1}). This fact, combined with rearrangement inequalities are the key points to prove the existence of ground state solutions. In the numerical part, the presence of the harmonic potential term is challenging. In fact, one can't take Fourier transfrom directly on both sides of the equation like \cite{R4} because we have nonlinear term. Different from \cite{R3}, we also can't use finite difference directly since fractional Laplacian is not a local term. Consequtently, we opted idea from \cite{R1} and use time splitting method. By our splitting, we can obtain specific solutions in each small step and also preserve the mass \eqref{E1.3}. For the ground state solutions, the classical Newton's method \cite{R3} is too slow because we have to deal with fractional Laplacian. To overcome this, we borrow idea from \cite{R2} and use normalized gradient flow (NGF) to find the ground state solutions. Moreover, for the case $\frac{2s}{N}\leq\sigma\leq \frac{2s}{N-2s}$, we have noticed that the energy in the original variational problem can not be bounded from below, therefore, we present a new constrained variational problem (\eqref{E6.15}-\eqref{E6.17}) to establish the existence of ground state solutions. 

The paper is organized as follows. In section 2, we give our main results about the existence of ground state solutions and orbital stability of standing waves. In section 3, we provide the proof of the existence. Then, in section 4, we discuss the orbital stability. In section 5, we use Split-Step Fourier Spectral method to solve \eqref{E1.1} numerically. In section 6, instead of using common iterative Newton's method, we use the NGF method to find ground states when $0<\sigma\leq \frac{2s}{N-2s}$. Finally, in section 7, we present our numerical results for the dynamics \eqref{E1.1} and compare them with other kinds of Schr\"odinger equations (\cite{R3,R4}).

\section{Main results}

We use a variational formulation to examine the solution to \eqref{E1.6}. First, note that if $\lambda=0$, we can find solutions $u(x)$ to \eqref{E1.6} from the critical points of the  functional  $\mathcal{J}: H_{s}(\mathbb{R}^{N})\longrightarrow \mathbb{R}$ defined as:
\begin{equation}\label{E2.1}\mathcal{J}(u)=\frac{1}{2}\|\nabla_{s}u\|_{2}^{2}+\frac{1}{2}\int_{\mathbb{R}^{N}}|x|^{2}|u|^{2}dx-\frac{1}{2\sigma+2}
\int_{\mathbb{R}^{N}}|u|^{2\sigma+2}dx,
\end{equation}
where $\|.\|_{L^{2}}$ is the $L^{2}$-norm and $\|\nabla_{s}u\|_{2}$ is defined by
\[
\|\nabla_{s}u\|_{2}^{2}=C_{N,s}\int_{\mathbb{R}^{N}}\int_{\mathbb{R}^{N}}\frac{|u(x)-u(y)|^{2}}{|x-y|^{N+2s}}dxdy,
\]
with some normalization constant $C_{N,s}$. 

We can derive \eqref{E2.1} by multiplying smooth enough test function $v(x)$ on both sides of \eqref{E1.6} and taking the integral over $x$. However, instead of directly finding the critical points of \eqref{E2.1}, we consider a reconstructed variational problem, which can help us to find solutions with different $\lambda$ and any energy. Specifically, for a fixed number $c>0$, we need to solve the following constrained minimization problem. 
\begin{equation}\label{E2.4}
  I_{c}=inf\left\{\mathcal{J}(u): u\in S_{c}\right\},
\end{equation}
with
\begin{equation}\label{E2.5}S_{c}=\left\{u\in \Sigma_{s}(\mathbb{R}^{N}): \int_{\mathbb{R}^{N}}|u|^{2}dx=c^{2}\right\},\end{equation}
where
\begin{equation}\label{E2.7}\Sigma_{s}({\mathbb{R}^{N}})=\left\{u\in H_{s}(\mathbb{R}^{N}):\|u\|_{\Sigma_{s}({\mathbb{R}^{N}})}:=\|u\|_{L^{2}({\mathbb{R}^{N}})}+\|\nabla_{s}u\|_{L^{2}({\mathbb{R}^{N}})}+
\|x u\|_{L^{2}({\mathbb{R}^{N}})}<\infty\right\}.\end{equation}
is a Hilbert space, with corresponding natural inner product.

We claim that for each minimizer $u(x)$ of the constrained minimization problem \eqref{E2.4}, there exists some $\lambda$ such that $(u(x),\lambda)$ is a solution to \eqref{E1.6}. To prove the claim, we first consider $\lambda$ as a Lagrange multiplier, then we define
\begin{equation}\label{E2.8}
\mathcal{J^*}(u)=\mathcal{J}(u)+\lambda(\|u\|^2_{L^2({\mathbb{R}^{N}})}-c^2).
\end{equation}
The minimizer to problem \eqref{E2.4} must be the critical point of \eqref{E2.8}, satisfying: 
\begin{equation}\label{EJ1}
\frac{\partial \mathcal{J^*}(u)}{\partial u}=0
\end{equation}
and
\begin{equation}\label{EJ2}
\frac{\partial \mathcal{J^*}(u)}{\partial \lambda}=0,
\end{equation}
where \eqref{EJ1} implies \eqref{E1.6} and \eqref{EJ2} implies \eqref{E2.5}. In this paper, we will mainly focus on the minimizers of problem \eqref{E2.4}. The following theorem discusses the existence of such minimizers.

\begin{theorem}\label{T1} If $0<\sigma<\frac{2s}{N}$, then \eqref{E2.4} admits a nonnegative, radial and radially decreasing minimizer.
\end{theorem}

\begin{remark}\label{R2.1} The condition $0<\sigma<\frac{2s}{N}$ is important in our proof of the existence of minimizers. For the critical case $\sigma=\frac{2s}{N}$, we were able to obtain interesting results (section 7).
\end{remark}

After we construct the ground state solutions, we further investigate their stability. By the definition of \eqref{E1.5}, the ground state solution moves around a circle when time changes. Therefore we consider and prove the orbital stability of ground state solution (Def \ref{D4.1}).

\begin{theorem}\label{T2}
 Suppose that $0<\sigma<\frac{2s}{N}$ and \eqref{E1.1} has a unique solution with conserved mass \eqref{E1.3} and energy \eqref{E1.31}, then the ground state solutions constructed in Theorem \ref{T1} are orbitally stable.
 \end{theorem}

\section{The minimization problem}
In this section, we will establish the existence of ground state solutions of \eqref{E1.6}, the main difficulty comes from the lack of compactness of the Sobolev embeddings. Usually, at least when potential in \eqref{E1.1} is radially
symmetric and radially increasing, such a difficulty is overcame by considering the appropriate function space. More precisely, we have

\begin{lemma}\label{L3.1}
Let $2\leq p< \frac{2N}{N-2s}$, then  the  embedding $\Sigma_{s}({\mathbb{R}^{N}})\subset L^p(\mathbb{R}^N)$ is compact.
\end{lemma}
\begin{pf}
For any $u(x)\in\Sigma_s$, $\|u\|_{H_s({\mathbb{R}^{N}})}\leq \|u\|_{\Sigma_s({\mathbb{R}^{N}})}$, which implies $\Sigma_s({\mathbb{R}^{N}})$ can be embedding into $H_s$. On the other hand, by Sobolov embedding theorem, $H_s({\mathbb{R}^{N}})$ can be compactly embedded into $L^p(\mathbb{R}^N)$ for $2< p< \frac{2N}{N-2s}$. Therefore, $\Sigma_s({\mathbb{R}^{N}})$ can also be compactly embedded into $L^p(\mathbb{R}^N)$ for $2< p< \frac{2N}{N-2s}$.

Second, when $p=2$, choose $R>0$, then for any $u(x)\neq 0\in \Sigma_s({\mathbb{R}^{N}})$, we have
\begin{equation}\label{E3.0}
\int_{|x|\geq R}|u|^2dx\leq |R|^{-2}\int_{|x|\geq R}|x|^2|u|^2dx< |R|^{-2}\|u\|_{\Sigma_s({\mathbb{R}^{N}})}
\end{equation}

By the classical Sobolev embedding theorem, for any fixed $R$, $H_{s}(|x|<R)$ is compactly embedded in $L^2(|x|<R)$. Therefore, for any bounded sequence in $\Sigma_s$, we choose $R_n>0\rightarrow \infty$ and for each $n$, pick out the subsequence that converges in $L^2(|x|<R_n)$ from former convergence sequence in $L^2(|x|<R_{n-1})$, finally using the diagonal method combined with \eqref{E3.0} we find the convergence sequence in $L^2(\mathbb{R}^N)$.
\end{pf}

Then we have a lemma showing the existence of $I_{c}$ and boundedness of minimizing sequence.
\begin{lemma}\label{L3.2}
If $0<\sigma<\frac{2s}{N}$, then $I_{c}>-\infty$ and all minimizing sequences of \eqref{E2.4} are bounded in $\Sigma_{s}(\mathbb{R}^{N})$.
\end{lemma}
\begin{proof}
First, we prove that $\mathcal{J}(u)$ is bounded from below. Using the fractional Gagliardo-Nirenberg inequality \cite{HOW}, we certainly have
\begin{equation}\label{E3.1}
\|u\|_{2\sigma+2}\leq K\|u\|_{2}^{1-\theta}\|\nabla_{s}u\|_{2}^{\theta},
\end{equation}
for some positive constant $K$, where $\theta=\frac{Ns}{2s(\sigma+1)}.$\\

On the other hand, let $\epsilon>0$, and  $p, q>1$ such that $\frac{1}{p}+\frac{1}{q}=1$, then, using  Young's inequality, one gets
\begin{equation}\label{E3.2}
\|u\|_{2}^{(2\sigma+2)(1-\theta)}\|\nabla_{s}u\|_{2}^{\theta(2\sigma+2)}\leq\frac{1}{p}\epsilon^{p}\|\nabla_{s}u\|_{2}^{p\theta(2\sigma+2)}+\frac{1}{q\epsilon^{q}}\|u\|_{2}^{q(1-\theta)(2\sigma+2)}.
\end{equation}
Combining \eqref{E3.1} and \eqref{E3.2}, we obtain  for any $u\in S_{c}$,
\begin{equation}\label{E3.3}
\int_{\Omega}|u(x)|^{2\sigma+2}dx\leq \frac{\epsilon^{p}K^{2\sigma+2}}{p}\|\nabla_{s}u\|_{2}^{2}+ \frac{K^{2\sigma+2}}{q\epsilon^{q}}c^{2q(1-\theta)(1+\sigma)},
\end{equation}
where $p=\frac{1}{\theta(1+\sigma)}=\frac{2s}{N\sigma},\ q=\frac{1}{1-\theta(1+\sigma)}$.\\
Hence, from (\ref{E3.3}) we get:
\begin{eqnarray}
  \mathcal{J}(u)&=&\frac{1}{2}\|\nabla_{s}u\|_{2}^{2}+\frac{1}{2}\int_{\mathbb{R}^{N}}|x|^{2}|u|^{2}dx-\frac{1}{2\sigma+2}
\int_{\mathbb{R}^{N}}|u|^{2\sigma+2}dx \\
   &\geq& \frac{1}{2}\|\nabla_{s}u\|_{2}^{2}+\frac{1}{2}\int_{\mathbb{R}^{N}}|x|^{2}|u|^{2}dx-\frac{1}{2\sigma+2}\left(\frac{\epsilon^{p}K^{2\sigma+2}}{p}\|\nabla_{s}u\|_{2}^{2}+ \frac{K^{2\sigma+2}}{q\epsilon^{q}}c^{2q(1-\theta)(1+\sigma)}\right) \\
   &\geq& \left(\frac{1}{2}-\frac{\epsilon^{p}K^{2\sigma+2}}{2p(\sigma+1)}\right)\|\nabla_{s}u\|_{2}^{2}+\frac{1}{2}\int_{\mathbb{R}^{N}}|x|^{2}|u|^{2}dx-
   \frac{K^{2\sigma+2}c^{2q(1-\theta)(1+\sigma)}}{2q(1+\sigma)\epsilon^{q}}.\label{E3.6}
\end{eqnarray}

Then we choose $\epsilon$ small enough in \eqref{E3.6} to make $\left(\frac{1}{2}-\frac{\epsilon^{p}K^{2\sigma+2}}{2p(\sigma+1)}\right)>0$, which implies that $I_{c}>-\infty$ and that for all minimizing sequences $\{u_n\}$, $\mathcal{J}(u_n)$ is bounded from above, which implies $\{u_n\}$ is bounded in $\Sigma_{s}(\mathbb{R}^{N})$ by \eqref{E3.6}.
\end{proof}

Now, we can use compactness(Lemma \ref{L3.1}) and boundedness(Lemma \ref{L3.2}) to prove our existence Theorem \ref{T1}.

\begin{proof}

Let $\{u_{n}\}$ be a minimizing sequence of \eqref{E2.4}. By Lemma \ref{L3.2}, $\{u_{n}\}$ is bounded in $\Sigma_{s}(\mathbb{R}^{N})$. Up to a subsequence, there exists $u$ such that  $u_{n}$ converges weakly to $u$ in $\Sigma_{s}(\mathbb{R}^{N})$.

Since $2\sigma+2<\frac{2N}{N-2s}$ and  $\Sigma_{s}(\mathbb{R}^{N})$ is compactly embedded in $L^{p}(\mathbb{R}^{N})$ for any $p$ such that $2\leq p<\frac{2N}{N-2s}$, we can further prove that $u_{n}$ will converge strongly to $u$ in $L^2({\mathbb{R}^{N}})$ and $L^{2\sigma+2}({\mathbb{R}^{N}})$ (Lemma \ref{L3.1}). In particular, $u_n\rightarrow u$ in $L^2({\mathbb{R}^{N}})$ implies $u\in \mathcal{S}_c$.

 On the other hand, thanks to the lower semi-continuity, we have $\|xu\|_{L^2({\mathbb{R}^{N}})}+\|\nabla_{s}u\|_{L^2({\mathbb{R}^{N}})}\leq \liminf_{n\rightarrow\infty} \|\nabla_{s}u_{n}\|_{L^2({\mathbb{R}^{N}})}+\|xu_n\|_{L^2({\mathbb{R}^{N}})}$. Therefore
\begin{equation}\label{E3.7}
\mathcal{I}_{c}\leq \mathcal{J}(u)\leq \liminf_{n\rightarrow\infty} \mathcal{J}(u_{n})=\mathcal{I}_{c},
\end{equation}
which yields $u$ is a minimizer.

The second step consists in constructing a nonnegative, radial and radially decreasing minimizer. First, note that:
\begin{equation}\label{E3.10}
\|\nabla_{s}|u|\|_{L^2({\mathbb{R}^{N}})}\leq \|\nabla_{s}u\|_{L^2({\mathbb{R}^{N}})},
\end{equation}
which implies $\mathcal{J}(|u|)\leq\mathcal{J}(u)$. Then we use the Schwarz symmetrization \cite{R5}. We construct a
symmetrization function $u^{\ast}$, which is a radially-decreasing function from $\mathbb{R}^{N}$ into $\mathbb{R}$ with the property
that $$meas\left\{x \in \mathbb{R}^{N}: u(x) > \mu\right\} = meas\left\{x \in \mathbb{R}^{N}: u^{\ast}(x) > \mu\right\}\;\mbox{ for any }\;\mu > 0.$$ It's well-known \cite{R5} that
\begin{equation}\label{E3.8}
\left\{
    \begin{array}{ll}
      \int_{\mathbb{R}^N}|u|^{2\sigma+2}dx= \int_{\mathbb{R}^N}|u^{\ast}|^{2\sigma+2}dx\\
      \int_{\mathbb{R}^N}|u|^{2}dx= \int_{\mathbb{R}^N}|u^{\ast}|^{2}dx,
    \end{array}
  \right.
\end{equation}

Besides, from \cite{R3},\cite{R8}, we also have
\begin{equation}\label{E3.9}
\left\{
    \begin{array}{ll}
      \|\nabla_{s}u^{\ast}\|_{L^2(\mathbb{R}^N)}\leq \|\nabla_{s}u\|_{L^2(\mathbb{R}^N)}\\
      \int_{\mathbb{R}^N}|x|^{2}|u^{\ast}|^{2}dx\leq \int_{\mathbb{R}^N}|x|^{2}|u|^{2}dx.
    \end{array}
  \right.
\end{equation}

Combining (\ref{E3.8}) and (\ref{E3.9}), we obtain
$$\mathcal{J}(|u|^{\ast})\leq \mathcal{J}(|u|)\leq\mathcal{J}(u),\quad \text{for any}\ u\in\Sigma_s(\mathbb{R}^N)$$

\end{proof}
\begin{remark}\label{R3.1}
By \eqref{E3.7}, and weakly convergence, we can also see $\|xu\|_{L^2({\mathbb{R}^{N}})}+\|u\|_{L^2({\mathbb{R}^{N}})}+\|\nabla_{s}u\|_{L^2({\mathbb{R}^{N}})}=\lim_{n\rightarrow\infty}\|\nabla_{s}u_{n}\|_{L^2({\mathbb{R}^{N}})}+\|u_n\|_{L^2({\mathbb{R}^{N}})}+\|xu_n\|_{L^2({\mathbb{R}^{N}})}$, which implies there is a minimizing subseqence $u_{n_k}$ converging to $u$ in $\Sigma_s$.
\end{remark}

\begin{remark}\label{R3.2} If $u\in\Sigma_s(\mathbb{R}^N)$ is a minimizer to \eqref{E2.4}, we must have 
\begin{equation}\label{equalcase}
\mathcal{I}_c=\mathcal{J}(u)=\mathcal{J}(|u|)=\mathcal{J}(|u|^*).
\end{equation}
By \eqref{E3.9}, \eqref{equalcase} implies
\begin{align}
&\|\nabla_{s}u\|_{L^2(\mathbb{R}^N)}=\|\nabla_{s}|u|\|_{L^2(\mathbb{R}^N)}\label{equal1},\\
&\int_{\mathbb{R}^N}|x|^{2}(|u|^{\ast})^{2}dx=\int_{\mathbb{R}^N}|x|^{2}|u|^{2}dx\label{equal2}.
\end{align}
By \cite{R3}, \eqref{equal1} implies $u=|u|$ a.e. and \eqref{equal2} implies $|u|=|u|^*$ a.e.. 
\end{remark}

\section{Orbital stability}

In this section, we will deal with the orbital stability of the ground state solutions. Let us introduce the appropriate Hilbert space:
\begin{equation*}
\label{E4.4}\widetilde{\Sigma}_{s}(\mathbb{R}^{N}):=\left\{\omega=u+iv:(u,v)\in\Sigma_{s}(\mathbb{R}^{N})\times \Sigma_{s}(\mathbb{R}^{N})\right\},
\end{equation*}
equipped with the norm $\|\omega\|^{2}_{\widetilde{\Sigma}_{s}(\mathbb{R}^{N})}=\|u\|^{2}_{\Sigma_{s}(\mathbb{R}^{N})}+\|v\|^{2}_{\Sigma_{s}(\mathbb{R}^{N})}$, which is a Hilbert space.

In term of the new coordinates, the energy functional reads
\begin{equation*}
\label{E4.5}\widetilde{\mathcal{J}}(\omega)=\frac{1}{2}\|\nabla_{s}\omega\|_{2}^{2}+\frac{1}{2}\int_{\mathbb{R}^{N}}
|x|^2|\omega(x)|^{2}dx-\frac{1}{2\sigma+2}\int_{\mathbb{R}^{N}}|\omega|^{2\sigma+2}dx,
\end{equation*}
where $\|\nabla_s\omega\|_{{L^2({\mathbb{R}^{N}})}}^{2}=\|\nabla_{s}u\|_{{L^2({\mathbb{R}^{N}})}}^{2}+\|\nabla_{s}v\|_{{L^2({\mathbb{R}^{N}})}}^{2}$, we can also get $\widetilde{\mathcal{J}}(\omega)$ remains as a constant with time $t$ if $\omega(t,x,v)$ is a solution to $\eqref{E1.1}$.

Then, for all $c>0$, we set a similar constrained minimization problem
\begin{equation*}\label{E4.6}\widetilde{\mathcal{I}}_{c}=inf\left\{\widetilde{\mathcal{J}}(\omega),\;\omega\in \widetilde{\mathcal{S}}_{c}\right\},\end{equation*}
where $\mathcal{\widetilde{S}}_{c}$ is defined by:
 \begin{equation*}\label{E4.7}
 \mathcal{\widetilde{S}}_{c}=\left\{\omega\in\widetilde{\Sigma}_{s}(\mathbb{R}^{N}),\;\;\int_{\mathbb{R}^{N}}|\omega(x)|^{2}dx=c^{2}\right\}.
 \end{equation*}
We also introduce the following sets
\begin{equation*}\label{E4.8}\mathcal{O}_{c}=\left\{u\in\mathcal{S}_{c}:\mathcal{I}_{c}=\mathcal{J}(u)\right\},\quad\widetilde{\mathcal{O}}_{c}=\left\{\omega\in\widetilde{\mathcal{S}}_{c}:\widetilde{\mathcal{I}}_{c}=\widetilde{\mathcal{J}}(\omega)\right\}.
\end{equation*}

Proceeding as in \cite{haj2, haj1}, we have the following lemma:

\begin{lemma}\label{lem3} If $0<\sigma<\frac{2s}{N}$, then the following properties hold true:

\textbf{(i)} The energy functional $\mathcal{J}$ and $\widetilde{\mathcal{J}}$ are of class $C^{1}$ on $\Sigma_{s}(\mathbb{R}^{N})$ and $\widetilde{\Sigma}_{s}(\mathbb{R}^{N})$ respectively.

\textbf{(ii)}  There exists a constant $C > 0$ such that
$$\|\mathcal{J}'(u)\|_{\Sigma^{-1}_s(\mathbb{R}^{N})}\leq C\left(\|u\|_{\Sigma_{s}(\mathbb{R}^{N})}+ \|u\|_{\Sigma_{s}(\mathbb{R}^{N})}^{2\sigma+1}\right),\quad \|\widetilde{\mathcal{J}}'(\omega)\|_{\Sigma^{-1}_s(\mathbb{R}^{N})}\leq C\left(\|\omega\|_{\Sigma_{s}(\mathbb{R}^{N})}+ \|\omega\|_{\Sigma_{s}(\mathbb{R}^{N})}^{2\sigma+1}\right)$$.
\textbf{(iii)} All minimizing sequences for $\widetilde{\mathcal{I}}_{c}$ are bounded in $\widetilde{\Sigma}_{s}(\mathbb{R}^{N})$ and all minimizing
sequences for $\mathcal{I}_{c}$ are bounded in $\Sigma_{s}(\mathbb{R}^{N})$.

\textbf{(iv)} The mappings $c\longmapsto \mathcal{I}_{c},\,\widetilde{\mathcal{I}}_{c}$ are continuous.

\textbf{(v)} Any minimizing sequence of $\mathcal{I}_{c}$, $\widetilde{\mathcal{I}}_{c}$ are relatively compact in $\Sigma_{s}(\mathbb{R}^{N})$, $\widetilde{\Sigma_{s}}(\mathbb{R}^{N})$.

\textbf{(vi)} For any $c>0$,
\begin{equation*}\label{E4.9}
\mathcal{I}_c=\widetilde{\mathcal{I}}_c.
\end{equation*}

\end{lemma}

\begin{proof}
\textbf{(i)} We follow the steps of Proposition 2.3 \cite{haj1} by choosing $g(x,t)=-t^{\sigma}$. For any $u, v\in \Sigma_{s}(\mathbb{R}^{N})$, we can see the last term of functional
\begin{equation*}
-\int_{\mathbb{R}^{N}}|u(x)|^{2\sigma}u(x)v(x)dx
\end{equation*}
is of class $C^{1}$ on $\Sigma_{s}(\mathbb{R}^{N})$. Then by the definition of $\Sigma_{s}(\mathbb{R}^{N})$ (see \eqref{E2.7}), the first two terms of the functional are of class $C^1$ on $\Sigma_{s}(\mathbb{R}^{N})$.

\textbf{(ii)}  From \textbf{(i)}, $\mathcal{J}$ is of class $C^{1}$ on $\Sigma_{s}(\mathbb{R}^{N})$. Moreover, for all $u, v\in \Sigma_{s}(\mathbb{R}^{N})$, we have
\begin{eqnarray*}
 \prec \mathcal{J}'(u),v\succ &=& C_{N,S}\int_{\mathbb{R}^{N}}\int_{\mathbb{R}^{N}}\frac{|u(x)-u(y)||v(x)-v(y)|}{|x-y|^{N+2s}}dxdy+\int_{\mathbb{R}^{N}}|x|^{2}u(x)v(x)dx \\
   && -\int_{\mathbb{R}^{N}}|u(x)|^{2\sigma}u(x)v(x)dx.
\end{eqnarray*}

For the last term, by H\"older's inequality
\begin{equation*}\label{E4.10}
\int_{\mathbb{R}^{N}}|u(x)|^{2\sigma}u(x)v(x)dx\leq \|u\|_{L^{2\sigma+2}(\mathbb{R}^{N})}^{2\sigma+1}\|v\|_{L^{2\sigma+2}(\mathbb{R}^{N})}
\end{equation*}
Therefore, there exists $C>0$ such that
\begin{equation*}\label{E4.11}
\|\mathcal{J}'(u)\|_{\Sigma^{-1}_s(\mathbb{R}^{N})}\leq C\left(\|u\|_{\Sigma_{s}(\mathbb{R}^{N})}+ \|u\|_{\Sigma_{s}(\mathbb{R}^{N})}^{2\sigma+1}\right).
\end{equation*}

\textbf{(iii)} This is a direct result of Lemma \eqref{L3.2}.

\textbf{(iv)} Let $c>0$ and let $\{c_{n}\}\subset (0,\infty)$ such that $c_{n}\rightarrow c$. It suffices to prove that $\mathcal{I}_{c_{n}}\rightarrow \mathcal{I}_{c}$. By the definition of $\mathcal{I}_{c_{n}}$, for any $n$ there exists $u_{n}\in S_{c_{n}}$ such that
\begin{equation}\label{E4.12}
 \mathcal{I}_{c_{n}}\leq \mathcal{J}(u_{n})<\mathcal{I}_{c_{n}}+\frac{1}{n}.
\end{equation}
From \textbf{(iii)}, there exists a constant $C_{1}>0$ such that for all $n$, we have  $$\|u_{n}\|_{\Sigma_{s}(\mathbb{R}^{N})}\leq C_{1},\;\;\forall\;n\in \mathbb{N}.$$
Set $v_{n}=\frac{c}{c_{n}}u_{n}$, then, for all $n\in \mathbb{N}$,  we have
\begin{equation*}\label{E4.13}
  v_{n}\in S_{c}\;\;\mbox{and}\;\;\|u_{n}-v_{n}\|_{\Sigma_{s}(\mathbb{R}^{N})}=\left|1-\frac{c}{c_{n}}\right|\|u_{n}\|_{\Sigma_{s}(\mathbb{R}^{N})}\leq C_{1}\left|1-\frac{c}{c_{n}}\right|,
\end{equation*}
which implies
\begin{equation}\label{E4.14}
\|u_{n}-v_{n}\|_{\Sigma_{s}(\mathbb{R}^{N})}\leq C_{1}+1\;\;\mbox{for} \;n\;\mbox{ large enough }.
\end{equation}

We deduce by part \textbf{(ii)} that there exists a positive constant $K:=K(C_{1})$ such that
\begin{equation}\label{E4.15}
\|\mathcal{J}'(u)\|_{\Sigma^{-1}_s(\mathbb{R}^{N})}\leq K,\;\mbox{for all}\;u\in \Sigma_{s}(\mathbb{R}^{N})\;\mbox{with}\;\|u\|_{\Sigma_{s}(\mathbb{R}^{N})}\leq 2C_1+1.
\end{equation}
From(\ref{E4.14}) and (\ref{E4.15}) we obtain
\begin{eqnarray}\label{E4.16}
  \left|\mathcal{J}(v_{n})-\mathcal{J}(u_{n})\right| &=& \left|\int_{0}^{1}\frac{d}{dt}\mathcal{J}\left(t v_{n}+(1-t)u_{n}\right)dt\right| \nonumber \\
   &\leq & \displaystyle\sup_
   {\|u\|_{\Sigma_{s}(\mathbb{R}^{N})}\leq 2 C_{1}+1}\|\mathcal{J}'(u)\|_{\Sigma^{-1}_s(\mathbb{R}^{N})}\|v_{n}-u_{n}\|_{\Sigma_{s}(\mathbb{R}^{N})} \nonumber\\
   &\leq&  K C_{1}\left|1-\frac{c}{c_{n}}\right|.
\end{eqnarray}
Then, from (\ref{E4.12}) and (\ref{E4.16}), we obtain
\begin{eqnarray*}
   \mathcal{I}_{c_{n}} &\geq& \mathcal{J}(u_{n})-\frac{1}{n} \\
   &\geq& \mathcal{J}(v_{n})- K C_{1}\left|1-\frac{c}{c_{n}}\right|-\frac{1}{n}\\
   &\geq&  \mathcal{I}_{c}- K C_{1}\left|1-\frac{c}{c_{n}}\right|-\frac{1}{n}
\end{eqnarray*}
Combining this with the fact that $\displaystyle\lim_{n\rightarrow \infty}c_{n}=c$, it yields
\begin{equation}\label{E4.17}
  \displaystyle \liminf_{n\rightarrow \infty}\mathcal{I}_{c_{n}}\geq\mathcal{I}_{c}.
\end{equation}

Now, from Lemma \eqref{L3.2} and by the definition of $\mathcal{I}_{c}$, there exists a positive constant $C_{2}$ and  a sequence $\{u_{n}\}\subset \mathcal{S}_{c}$ such that
\begin{equation*}\label{E4.171}
  \|u_{n}\|_{\Sigma_{s}(\mathbb{R}^{N})}\leq C_{2}\;\;\mbox{and}\;\;\displaystyle\lim_{n\rightarrow \infty} \mathcal{J}(u_{n})=\mathcal{I}_{c}.
\end{equation*}
Set $v_{n}=\frac{c_{n}}{c}u_{n}$, then $v_{n}\in S_{c_{n}}$, there exists a constant $L=L(C_{2})$ such that
$$\|v_{n}-u_{n}\|_{\Sigma_{s}(\mathbb{R}^{N})}\leq C_{2}\left|1-\frac{c_{n}}{c}\right|\;\;\mbox{and}\;\;| \mathcal{J}(v_{n})- \mathcal{J}(u_{n})|\leq L\,C_{2}\left|1-\frac{c_{n}}{c}\right|.$$
Combining this with (\ref{E4.12}), we obtain
$$\mathcal{I}_{c_{n}}\leq \mathcal{J}(v_{n})\leq \mathcal{J}(u_{n})+L\,C_{2}\left|1-\frac{c_{n}}{c}\right|.$$
Since $\displaystyle\lim_{n\rightarrow \infty}c_{n}=c$, we have
\begin{equation}\label{E4.18}
  \displaystyle \limsup_{n\rightarrow \infty}\mathcal{I}_{c_{n}}\leq\mathcal{I}_{c}.
\end{equation}
It follows from \eqref{E4.17} and \eqref{E4.18} that
$$\displaystyle \lim_{n\rightarrow \infty}\mathcal{I}_{c_{n}}=\mathcal{I}_{c}.$$

\textbf{(v)} This is a direct result of Remark \ref{R3.1}.

\textbf{(vi)} First, we can see $\Sigma_s(\mathbb{R}^{N})\subset\widetilde{\Sigma}_s(\mathbb{R}^{N})$, and any $\omega\in\Sigma_s(\mathbb{R}^{N})$, we have 
\[
\widetilde{\mathcal{J}}(\omega)=\mathcal{J}(\omega)
\]
which implies 
\begin{equation}\label{I1}
\mathcal{I}_c \geq \widetilde{\mathcal{I}}_c.
\end{equation}
Second, for any $\omega\in\widetilde{\Sigma}_s$, we have
\[
\|\nabla_{s}\omega\|^{2}_{L^2(\mathbb{R}^{N})}\geq\|\nabla_{s}|\omega|\|^{2}_{L^2(\mathbb{R}^{N})},
\]
which implies 
\[\widetilde{\mathcal{J}}(\omega)\geq \widetilde{\mathcal{J}}(|\omega|)=\mathcal{J}(\omega)\geq \mathcal{I}_c,\quad \forall\omega\in\widetilde{\Sigma}_s(\mathbb{R}^{N}),
\]
from which we can easily obtain 
\begin{equation}\label{I2}
\widetilde{\mathcal{I}}_c\geq \mathcal{I}_c.
\end{equation}
Combine \eqref{I1} and \eqref{I2}, we finally have $\widetilde{\mathcal{I}}_c=\mathcal{I}_c.$
\end{proof}

Now, for a fixed $c > 0$, we use the following definition of stability (see \cite{R7})

\begin{definition}\label{D4.1} We say that $\widetilde{\mathcal{O}}_{c}$ is stable if\\
\begin{itemize}
  \item $\widetilde{\mathcal{O}}_{c}$ is not empty.
  \item For all  $\omega_0\in\widetilde{\mathcal{O}}_{c}$ and $\varepsilon>0,$ there exists $\delta>0$ such that for all $\psi_{0}\in \widetilde{\Sigma}_s(\mathbb{R}^{N})$, we have
  $$\|\omega_0-\psi_{0}\|_{\widetilde{\Sigma}_{s}(\mathbb{R}^{N})}<\delta\;\Longrightarrow\;\displaystyle
  \inf_{\omega\in\widetilde{\mathcal{O}}_{c}}\|\omega-\psi\|_{\widetilde{\Sigma}_{s}(\mathbb{R}^{N})}<\varepsilon,$$
  \end{itemize}
  where $\psi$ denotes the solution of (\ref{E1.1}) corresponding to the initial data $\psi_{0}$.
\end{definition}

If $\widetilde{\mathcal{O}}_{c}$ is stable, we say the ground state solutions in $\widetilde{\mathcal{O}}_{c}$ are orbitally stable. The following theorem states the orbital stability of $\widetilde{\mathcal{O}}_c$.

\begin{theorem}\label{T4.1}
 Suppose that $0<\sigma<\frac{2s}{N}$, and (\ref{E1.1}) with initial data $\psi_{0}
 \in\widetilde{\Sigma}_s(\mathbb{R}^{N})$ has the unique solution $\psi(t,x)\in \widetilde{\Sigma}_s(\mathbb{R}^{N})$ with
 \begin{equation}\label{E4.19}
 \|\psi(t,.)\|_{L^2(\mathbb{R}^{N})}=\|\psi_{0}(t,.)\|_{L^2(\mathbb{R}^{N})}\;\;\mbox{and}\;\widetilde{\mathcal{J}}(\psi(t,.))=\widetilde{\mathcal{J}}(\psi_{0}(t,.)),
 \end{equation}
 then $\widetilde{\mathcal{O}}_{c}$ is stable.
 \end{theorem}

\begin{proof}
The proof is by contradiction: Suppose that $\widetilde{\mathcal{O}}_{c}$ is not stable, then there exists  $\epsilon_{0}>0,\;\omega_0\in\widetilde{\mathcal{O}}_{c}$ and a sequence $\Phi_{0}^{n}\in \widetilde{\Sigma}_s(\mathbb{R}^{N})$ such that $\|\omega_0-\Phi^{n}_{0}\|_{\widetilde{\Sigma}_{s}(\mathbb{R}^{N})}\rightarrow 0\;\mbox{as}\;n\rightarrow \infty$, but
\begin{equation}\label{E4.20}
\displaystyle
  \inf_{Z\in\widetilde{\mathcal{O}}_{c}}\|\Phi^{n}(t_{n},.)-Z\|_{\widetilde{\Sigma}_{s}(\mathbb{R}^{N})}\geq\varepsilon,
\end{equation}
for some sequence $\{t_{n}\}\subset \mathbb{R}$, where $\Phi^{n}(t,.)$ is the unique solution of problem (\ref{E1.1}) corresponding to the initial condition $\Phi^{n}_{0}$.

Let $\omega_{n}=\Phi^{n}(t_{n},.)=(u_{n},v_{n})\in  \widetilde{\Sigma}_{s}(\mathbb{R}^{N})$. Then, since $\omega\in\widetilde{\mathcal{S}}_{c}$ and $\widetilde{\mathcal{J}}(\omega)=\widetilde{\mathcal{I}}_{c}$, it follows from the continuity of $\|.\|_{2}$ and $\widetilde{\mathcal{J}}$ in $\widetilde{\Sigma}_{s}(\mathbb{R}^{N})$ that

\begin{equation*}\label{E4.21}
\|\Phi^{n}_{0}\|_{2}\rightarrow c\;\mbox{and}\;\widetilde{\mathcal{J}}(\Phi^{n}_{0})\rightarrow\widetilde{\mathcal{I}}_{c},\ n\rightarrow\infty.
\end{equation*}

Thus, we deduce from (\ref{E4.19}) that
\begin{equation}\label{E4.22}
\|\omega_{n}\|_{2}=\|\Phi^{n}_{0}\|_{2}\rightarrow c\;\mbox{and}\;\widetilde{\mathcal{J}}(\omega_{n})=\widetilde{\mathcal{J}}(\Phi^{n}_{0})\rightarrow\widetilde{\mathcal{I}}_{c},\ n\rightarrow\infty.
\end{equation}

Since $\{\omega_{n}\}\subset  \widetilde{\Sigma}_{s}(\mathbb{R}^{N})$, it is easy to see that $\{|\omega_{n}|\}\subset  \Sigma_{s}(\mathbb{R}^{N})$. On the other hand, Lemma \ref{lem3} \textbf{(iii)} and proof of Lemma \ref{L3.2} imply that $\{\omega_{n}\}$ is bounded in $\widetilde{\Sigma}_{s}(\mathbb{R}^{N})$ and hence, by passing to a subsequence there exists $\omega=(u,v)\in\widetilde{\Sigma}_{s}(\mathbb{R}^{N})$ such that
\begin{equation}\label{E4.24}
u_{n}\rightharpoonup u,\;v_{n}\rightharpoonup v\quad\mbox{and}\quad\liminf_{n\rightarrow \infty}\|\nabla_{s}u_{n}\|_{L^2(\mathbb{R}^{N})}+\|\nabla_{s}v_{n}\|_{L^2(\mathbb{R}^{N})}\;\mbox{exists}.
\end{equation}
Now, by a straightforward computation we obtain
\begin{equation}\label{E4.25}
\widetilde{\mathcal{J}}(\omega_{n})-\widetilde{\mathcal{J}}(|\omega_{n}|)=\frac{1}{2}\|\nabla_{s}\omega_{n}\|^{2}_{L^2(\mathbb{R}^{N})}-\frac{1}{2}\|\nabla_{s}|\omega_{n}|\|^{2}_{L^2(\mathbb{R}^{N})}\geq0.
\end{equation}
Thus, we obtain
\begin{equation*}\label{E4.26}
\widetilde{\mathcal{I}}_{c}=\lim_{n\rightarrow\infty}\widetilde{\mathcal{J}}(\omega_{n})\geq \limsup_{n\rightarrow\infty}\mathcal{J}(|\omega_{n}|).
\end{equation*}
Besides, by \eqref{E4.22}, 
\begin{equation*}\label{E4.27}
\|\omega_{n}\|_{L^2(\mathbb{R}^{N})}^{2}=\||\omega_{n}|\|_{L^2(\mathbb{R}^{N})}^{2}=c_{n}^{2}\rightarrow c^{2}.
\end{equation*}
It follows from Lemma \ref{lem3} that we have
\begin{equation*}\label{E4.28}
\liminf_{n\rightarrow\infty} \mathcal{J}(|\omega_{n}|)\geq \liminf_{n\rightarrow\infty} \mathcal{I}_{c_{n}}=\mathcal{I}_{c}. 
\end{equation*}
Hence
\begin{equation}\label{E4.29}
\widetilde{\mathcal{I}}_{c}=\lim_{n\rightarrow \infty}\widetilde{\mathcal{J}}(\omega_{n})=\lim_{n\rightarrow \infty}\mathcal{J}(|\omega_{n}|)=\mathcal{I}_c. 
\end{equation}

It follows from (\ref{E4.24}), (\ref{E4.25}) and (\ref{E4.29}) that
\begin{equation*}\label{E4.281}
\displaystyle lim_{n\rightarrow \infty}\|\nabla_{s}u_{n}\|_{L^2(\mathbb{R}^{N})}^{2}+\|\nabla_{s}v_{n}\|_{L^2(\mathbb{R}^{N})}^{2}-\left|\nabla_{s}(u_{n}^{2}+v_{n}^{2})^{\frac{1}{2}}\right|^{2}=0,
\end{equation*}
which is equivalent to say that
\begin{equation}\label{E4.291}
\displaystyle\lim_{n\rightarrow \infty}\|\nabla_{s}w_{n}\|_{L^2(\mathbb{R}^{N})}^{2}=\displaystyle\lim_{n\rightarrow \infty}\|\nabla_{s}|w_{n}|\|_{L^2(\mathbb{R}^{N})}^{2}.
\end{equation}
The boundedness of $\omega_n$ in $\widetilde{\Sigma_s}(\mathbb{R}^{N})$ and (\ref{E4.291}) imply that $|w_{n}|$ is bounded in $\Sigma_{s}(\mathbb{R}^{N})$. By using a similar argument to Lemma \ref{L3.1}, there exists $\varphi\in \Sigma_{s}(\mathbb{R}^{N})$ such that
\begin{equation}\label{E4.30}
|\omega_n|\rightarrow\varphi\;\mbox{in}\;\Sigma_{s}(\mathbb{R}^{N})\;\mbox{and}\;\|\varphi\|_{L^2(\mathbb{R}^{N})}=c\;\mbox{with}\;\mathcal{J}(\varphi)=\mathcal{I}_{c}.
\end{equation}

Next, let us prove $\varphi=|\omega|=\left(|u|^2+|v|^2\right)^{1/2}$, Using \eqref{E4.24}, it follows that   
\[
u_{n}\longrightarrow u\;\mbox{ and }\;v_{n}\longrightarrow v\;\mbox{ in }\;L^{2}(B(0,R))\;\mbox{for all}\;R>0.
\]
Since $\left|(u_{n}^{2}+v_{n}^{2})^{\frac{1}{2}}-(u^{2}+v^{2})^{\frac{1}{2}}\right|\leq |u_{n}-u|^{2}+|v_{n}-v|^{2}$, then, one has
\[
(u_{n}^{2}+v_{n}^{2})^{\frac{1}{2}}\longrightarrow (u^{2}+v^{2})^{\frac{1}{2}}\;\mbox{in}\;L^{2}(B(0,R)).
\]
But $|\omega_n|=(u_{n}^{2}+v_{n}^{2})^{\frac{1}{2}}\longrightarrow \varphi\;\mbox{in}\;\Sigma_s\subset L^{2}(\mathbb{R}^N).$ Thus, we certainly have
\[
(u^{2}+v^{2})^{\frac{1}{2}}=|\omega|=\varphi.
\] 
This further implies 
\begin{equation}\label{omegaphi}
\|\omega\|_{L^2(\mathbb{R}^{N})}=\|\varphi\|_{L^2(\mathbb{R}^{N})}=c,\quad \|\omega\|_{L^{2\sigma+2}(\mathbb{R}^{N})}=\|\varphi\|_{L^{2\sigma+2}(\mathbb{R}^{N})}.
\end{equation}
and
\begin{equation}\label{B1}
\begin{aligned}
&\frac{1}{2}\int_{\mathbb{R}^{N}}
|x|^2|\omega|^{2}dx-\frac{1}{2\sigma+2}\int_{\mathbb{R}^{N}}|\omega|^{2\sigma+2}dx\\
=&\frac{1}{2}\int_{\mathbb{R}^{N}}
|x|^2|\varphi|^{2}dx-\frac{1}{2\sigma+2}\int_{\mathbb{R}^{N}}|\varphi|^{2\sigma+2}dx\\
=&\lim_{n\rightarrow\infty} \frac{1}{2}\int_{\mathbb{R}^{N}}
|x|^2|\omega_n|^{2}dx-\frac{1}{2\sigma+2}\int_{\mathbb{R}^{N}}|\omega_n|^{2\sigma+2}dx.
\end{aligned}
\end{equation}
Additionally, by the lower semi-continuity, we further have
\begin{equation}\label{B2}
\|\nabla\omega\|^2_{L^2(\mathbb{R}^{N})}\leq \liminf_{n\rightarrow\infty}\|\nabla\omega_n\|^2_{L^2(\mathbb{R}^{N})}.
\end{equation}
\eqref{B1} together with \eqref{B2} and $\omega\in\widetilde{\mathcal{S}}_c$, we finally obtain
\begin{equation*}
\widetilde{\mathcal{I}}_c\leq \widetilde{\mathcal{J}}(\omega)\leq \lim_{n\rightarrow\infty}\widetilde{\mathcal{J}}(\omega_n)=\widetilde{\mathcal{I}}_c,
\end{equation*}
which implies
\begin{equation}\label{Fstep}
\omega\in\widetilde{\mathcal{O}}_c\quad \text{and}\quad \|\nabla\omega\|^2_{L^2(\mathbb{R}^{N})}=\lim_{n\rightarrow\infty}\|\nabla\omega_n\|^2_{L^2(\mathbb{R}^{N})}.
\end{equation}
Therefore, combining \eqref{E4.30}, \eqref{omegaphi} and \eqref{Fstep}, we finally obtain
\begin{equation*}
\omega_n\rightarrow\omega\quad \text{in}\quad \widetilde{\Sigma}_s(\mathbb{R}^N),
\end{equation*}
which contradicts to \eqref{E4.20}.
\end{proof}

\section{Numerical method for Fractional NLS with harmonic potential}
In this section, we consider numerical methods to solve \eqref{E1.1} and introduce the Split-Step Fourier Spectral method.

First, we truncate \eqref{E1.1} into a finite computational domain $[-L, L]^N$ with periodic boundary conditions:
\begin{equation}\label{E5.1}
\left\{\begin{aligned}
  &i \psi_{t}=(-\Delta)^{s}\psi+|\textbf{x}|^{2}\psi-|\psi|^{2\sigma}\psi ,\ \; t>0\\
  &\psi(0,\textbf{x})=\psi_{0}(\textbf{x}),
\end{aligned}
\right.
\end{equation}
for $\textbf{x}\in[-L,L]^N$.

Let $\tau>0$ be the time step, and define the time sequence $t_n = n\tau$ for $n\geq0$ and the mesh size $h = 2L/J$, where $J$ is a
positive even integer. The spatial grid points are
\begin{equation}\label{E5.11}
\left(\textbf{x}^{\textbf{j}}\right)_{n} = -L + \left(\textbf{j}\right)_nh,\ 1\leq n\leq N,
\end{equation}
where $\textbf{j}$ is a $N$-dimension integer vector with each component between 0 and $J$.

Denote $\psi^n_{\textbf{j}}$ as the numerical approximation of the solution $\psi(\textbf{x}^{\textbf{j}},t_n)$.
By the definition of fractional Laplacian in \eqref{E1.2}, we use the Fourier spectral method for spatial discretization. Hence, we assume the
ansatz:
\begin{equation}\label{E5.2}
\psi(\textbf{x}^{\textbf{j}},t)=\sum_{\textbf{k}\in \textbf{K}}\widehat{\psi_{\textbf{k}}}(t)\exp(i\mu^{\textbf{k}}\textbf{x}^{\textbf{j}}),
\end{equation}
where $I=\left\{\textbf{k}\in \mathbb{R}^N|-J/2\leq \textbf{k}_l\leq J/2-1, 1\leq l\leq N\right\}$, $\left(\mu^{\textbf{k}}\right)_k=\textbf{k}_k\pi/L$, $1\leq k\leq N$.

Now, we introduce the Split-step Fourier Spectral method. The main idea of this method is to solve \eqref{E5.1} in two splitting steps from $t=t_n$ to $t=t_{n+1}$ :
\begin{align}
&i \psi_{t}=|\textbf{x}|^{2}\psi-|\psi|^{2\sigma}\psi\label{E5.3},\\
&i \psi_{t}=(-\Delta)^{s}\psi\label{E5.4}.
\end{align}

First, by multiplying $\psi^*$ on both sides of \eqref{E5.3} and subtracting it from its conjugate, we obtain $|\psi(\textbf{x},t)|=|\psi(\textbf{x},t_n)|$ for any $t\in[t_n,t_{n+1})$, therefore, \eqref{E5.3} can be simplified to
\begin{equation}\label{E5.5}
i \psi_{t}=|\textbf{x}|^{2}\psi-|\psi(\textbf{x},t_n)|^{2\sigma}\psi.
\end{equation}

Second, taking Fourier transform on both sides of \eqref{E5.4}, we get
\begin{equation}\label{E5.6}
i\frac{d\hat{\psi_{\textbf{k}}}
(t)}{dt}=|\mu^{\textbf{k}}|^{2s}\hat{\psi_\textbf{k}}(t).
\end{equation}

We use the second order Strang splitting method with \eqref{E5.5} and \eqref{E5.6} as follows:
\begin{align}
&\psi^{n,1}_{\textbf{j}}=\psi^n_{\textbf{j}}\exp(-i(|\textbf{x}^{\textbf{j}}|^2-|\psi^n_{\textbf{j}}|^{2\sigma})\tau/2)\\
&\psi^{n,2}_{\textbf{j}}=\sum_{\textbf{k}\in \textbf{K}}\widehat{\psi^{n,1}}_{\textbf{k}}\exp(-i|\mu^{\textbf{k}}|^{2s}\tau)\exp(i\mu^{\textbf{k}}\textbf{x}^{\textbf{j}})\\
&\psi^{n+1}_{\textbf{j}}=\psi^{n,2}_{\textbf{j}}\exp(-i(|\textbf{x}^{\textbf{j}}|^2-|\psi^{n,2}_{\textbf{j}}|^{2\sigma})\tau/2)
\end{align}
where $\textbf{j}$ comes from \eqref{E5.11} and $n\geq 0$. For $n=0$, initial condition \eqref{E5.1} is discretized as:
\begin{equation}\label{E5.7}
\psi^0_{\textbf{j}}=\psi_0(\textbf{x}^{\textbf{j}})
\end{equation}

This method has spectral-order accuracy in space and second order in time. Similar to \cite{R1}, this method preserves discrete mass corresponding to \eqref{E1.3} defined as
\begin{equation}\label{E5.8}
M^n=\left(h^N\sum_{\textbf{j}}|\psi^n_{\textbf{j}}|^2\right)^{1/2}.
\end{equation}

\section{Numerical method to solve ground state solutions}
To find ground state solutions, we have to solve the following equation corresponding to $u(x)$:
\begin{equation}\label{E6.1}
(-\Delta)^su+|x|^2u-|u|^{2\sigma}u=\lambda u\ \quad x\in\mathbb{R}^N.
\end{equation}

As discussed previously, for $\sigma<\frac{2s}{N}$, we can solve \eqref{E2.4}-\eqref{E2.7} to find  a solution to \eqref{E6.1}. In order to calculate the minimizer of $\mathcal{J}(u)$ in $S_c$, we use normalized gradient flow method (NGF) \cite{R2}. We first apply the steepest gradient decent method to the energy functional $\mathcal{J}(u)$ without constraint. Then we project the solution back onto the sphere $\mathcal{S}_c$ to make sure that the constraint $||u||_{L^2(\mathbb{R}^{N})} = c$ is satisfied.

Thus, for a given sequence of time $0 = t_0 < t_1 < ... < t_n$ with fixed time step $\tau$, we compute the approximated solution $u^{(n)}$ of the partial differential equation
\begin{equation*}\label{E6.10}
\frac{\partial u}{\partial t} = -\frac{\partial E(u)}{\partial u}
\end{equation*}
combined with the projection onto $\mathcal{S}_c$ at each step. Specifically,
\begin{equation*}\label{E6.11}
\left\{\begin{aligned}
&\frac{\partial \widetilde{u}}{\partial t}=-(-\Delta)^s\widetilde{u}-|x|^2\widetilde{u}+|\widetilde{u}|^{2\sigma}\widetilde{u},\quad \quad t_n<t<t_{n+1}\\
&\widetilde{u}(x,t_n)=u^{(n)}(x)\\
&u^{(n+1)}(x)=c\frac{\widetilde{u}(x,t_{n+1})}{||\widetilde{u}(\cdot,t_{n+1})||_{L^2(\mathbb{R}^{N})}}
\end{aligned}
\right..
\end{equation*}
Here, we use semi-implicity time discretization scheme:
\begin{equation*}\label{E6.12}
\left\{\begin{aligned}
&\frac{\widetilde{u}^{(n+1)}-\widetilde{u}^{(n)}}{\Delta t}=-(-\Delta)^s\widetilde{u}^{(n+1)}-|x|^2\widetilde{u}^{(n+1)}+|\widetilde{u}^{(n)}|^{2\sigma}\widetilde{u}^{(n+1)},\quad \quad t_n<t<t_{n+1}\\
&\widetilde{u}(x,t_n)=u^{(n)}(x)\\
&u^{(n+1)}(x)=c\frac{\widetilde{u}^{(n+1)}(x)}{||\widetilde{u}^{(n+1)}(\cdot)||_{L^2(\mathbb{R}^{N})}}
\end{aligned}
\right.
\end{equation*}
with
\begin{equation*}\label{E6.13}
\delta^{s}_{\textbf{x}}u|_{\textbf{j}}=\sum_{\textbf{k}\in \textbf{K}}|\mu^{\textbf{k}}|^{2s}\exp(i\mu^{\textbf{k}}\textbf{x}^{\textbf{j}})
\end{equation*}
to discretize fractional laplacian, where $\textbf{K}$, $\mu^{\textbf{k}}$ are defined in \eqref{E5.2} .

Therefore in each step, we solve :
\begin{equation}\label{E6.14}
\left\{\begin{aligned}
&\frac{\widetilde{u}^{(n+1)}_{\textbf{j}}-u^{(n)}_{\textbf{j}}}{\tau}=-\delta^{s}_{\textbf{x}}\widetilde{u}^{(n+1)}|_{\textbf{j}}-(|\textbf{x}^\textbf{j}|^2-|u^{(n)}_{\textbf{j}}|^{2\sigma})\widetilde{u}^{(n+1)}_{\textbf{j}}\\
&u^{(n+1)}_{\textbf{j}}=c\frac{\widetilde{u}^{(n+1)}_{\textbf{j}}}{M^n}
\end{aligned}
\right.,
\end{equation}
where $M^n$ is defined in \eqref{E5.8}, $\textbf{j}$ comes from \eqref{E5.11} and $n\geq 0$. For $n=0$, we guess a starting function and discretize it as \eqref{E5.7}.

We need to notice that we can only solve \eqref{E2.4} for $\sigma<\frac{2s}{N}$. If $\sigma\geq \frac{2s}{N}$, $\|u\|_{2\sigma+2}$ can not be bounded by $\|\cdot\|_{\Sigma_s}$, which will cause $I_c=-\infty$ in $\mathcal{S}_c$. However, we can use another constrained variational form to find standing waves to \eqref{E1.1} for $\frac{2s}{N}\leq \sigma<\frac{2s}{N-2s}$. 

For $\sigma<\frac{2s}{N-2s}$, we define the following constrained minimization problem:
\begin{equation}\label{E6.15}
  L_{c}=inf\left\{\mathcal{K}(u): u\in T_{c}\right\},
\end{equation}
with
\begin{equation*}\label{E6.16}T_{c}=\left\{u\in \Sigma_{s}(\mathbb{R}^{N}): \|u\|_{2\sigma+2}=c\right\},\end{equation*}
\begin{equation}\label{E6.17}\mathcal{K}(u)=\frac{1}{2}\|\nabla_{s}u\|_{L^{2}}^{2}+\frac{1}{2}\int_{\mathbb{R}^{N}}|x|^{2}|u(x)|^2dx+\frac{\omega}{2}\int_{\mathbb{R}^{N}}|u(x)|^2dx,\end{equation}
where $\Sigma_s(\mathbb{R}^{N})$ is defined as \eqref{E2.7}. 

For $u\in T_c$, we have the estimate
\begin{equation*}
\begin{aligned}
\int_{\mathbb{R}^{N}}|u(x)|^2dx&\leq |R|^{-2}\int_{|x|>R}|x|^2|u|^2dx+\int_{|x|<R}|u|^2dx\\
&\leq  |R|^{-2}\int_{|x|>R}|x|^2|u|^2dx+C_{R,\sigma}\|u\|^2_{2\sigma+2},
\end{aligned}
\end{equation*}
where $C_{R,\delta}$ depends on $R,\delta$. Hence, for $\omega<0$, if we choose $R$ large enough to make $\left(\frac{1}{2}+\frac{\omega}{R^2}\right)>0$, we get
\begin{equation}\label{K}
\mathcal{K}(u)\geq \frac{1}{2}\|\nabla_{s}u\|_{L^{2}}^{2}+\left(\frac{1}{2}+\frac{\omega}{R^2}\right)\int_{\mathbb{R}^{N}}|x|^{2}|u(x)|^2dx+\frac{\omega}{2} C_{R,\sigma}c^2>-\infty
\end{equation}
for $u\in T_c$. Besides, if $\omega>0$, \eqref{K} is greater than $0$. Similar to Lemma \ref{L3.2} and Theorem \ref{T1}, there exists a local minimizer for \eqref{E6.15} with any $\omega$ and $c$. 

Now we see $\lambda$ as the Lagrange multiplier like \eqref{E2.8} but with a different functional
\begin{equation*}
\mathcal{K}^*(u)=\mathcal{K}(u)-\frac{\lambda}{2\sigma+2}\left(\|u\|^{2\sigma+2}_{2\sigma+2}-c^{2\sigma+2}\right),
\end{equation*}
then we can have the critical points $u^*$ and $\lambda^*$ satisfy
\begin{equation}\label{ustar}
(-\Delta)^su^*+|x|^2u^*-\lambda|u^*|^{2\sigma}u^*+\omega u^*=0,\quad x\in\mathbb{R}^N
\end{equation}
by $\frac{\partial\mathcal{K}^*(u)}{\partial u}=0$. If $\omega>0$, by multiplying $\overline{u^*}$ on both sides of \eqref{ustar} and taking the integral, we can see 
\begin{equation}\label{KP1}
\mathcal{K}^*(u^*)=\mathcal{K}(u^*)=\frac{\lambda}{2}c^{2\sigma+2}>0,
\end{equation}
which implies $\lambda>0$. Therefore, when $\omega>0$, we can define $u_{\omega,c}=\lambda^{1/2\sigma}u^*$ and obtain
\begin{equation}\label{KP2}
(-\Delta)^su_{\omega,c}+|x|^2u_{\omega,c}-|u_{\omega,c}|^{2\sigma}u_{\omega,c}+\omega u_{\omega,c}=0,\quad x\in\mathbb{R}^N,
\end{equation}
which means $u_{\omega,c}e^{i\omega t}$ is one standing wave solution to \eqref{E1.1}. We need to mention \eqref{KP1},\eqref{KP2} actually showed that we can find a ground state solution by solving \eqref{E6.17} if $L_c=\mathcal{K}(u^*)>0$. In fact, we have $\mathcal{K}(u^*)>0$ with $\omega<0$ but not very small. This is related with the smallest eigenvalue of $\left(-\Delta\right)^s+|x|^2$ (\cite{R3}). 

Now, for $\frac{2s}{N}\leq\sigma<\frac{2s}{N-2s}$, we use NGF method and semi-implicity time discretization scheme to solve constrained problem \eqref{E6.15}. Similar to \eqref{E6.14}, the scheme is 
\begin{equation}\label{NGF2}
\left\{\begin{aligned}
&\frac{\widetilde{u}^{(n+1)}_{\textbf{j}}-u^{(n)}_{\textbf{j}}}{\tau}=-\delta^{s}_{\textbf{x}}\widetilde{u}^{(n+1)}|_{\textbf{j}}-(|\textbf{x}^\textbf{j}|^2+u^{(n)}_{\textbf{j}})\widetilde{u}^{(n+1)}_{\textbf{j}}\\
&u^{(n+1)}_{\textbf{j}}=c\frac{\widetilde{u}^{(n+1)}_{\textbf{j}}}{M^n_{2\sigma+2}}
\end{aligned}
\right.,
\end{equation}
where $M^n_{2\sigma+2}$ is discrete $L^{2\sigma+2}$ norm
\[
M^n_{2\sigma+2}=\left(h^N\sum_{\textbf{j}}|u^n_{\textbf{j}}|^{2\sigma+2}\right)^{1/(2\sigma+2)}
\]
and $\textbf{j}$ comes from \eqref{E5.11} and $n\geq 0$.  For $n=0$, we guess a starting function and discretize it as \eqref{E5.7}.

\section{Numerical results}
In this section, we show some numerical results, which can help us understand the ground state solution and also illustrate theoretical results. We have mainly investigated: 1. \textit{Ground state solutions with different $s$.} 2. \textit{Ground state solutions with non-symmetric potentials.} 3. \textit{Stability and dynamcis.} 

\subsection{Numerical results of ground state solutions}
First, we solve \eqref{E2.4} numerically by \eqref{E6.14} in one dimension $N=1$ for the case $s=0.8$ and $\sigma=1$ to obtain a ground state solution $u_0(x)$. From figure \ref{F1.1} we see the $u_0(x)$ is radially decreasing as Remark \ref{R3.2}. 

Second, we put $u_0(x)$ into the \eqref{E1.1} as initial condition and investigate time evolution of standing waves (Figure \ref{F1.2}-\ref{F1.4}). As expected, we see $|\psi(x,t)|$ is conserved and the real and imaginary part of solution change periodicly with time $t$.
 \begin{figure}[H]
  \centering
 \subfigure[Ground state solution $u_0(x)$]
    {\label{F1.1}\includegraphics[width=0.4\linewidth,height=3cm]{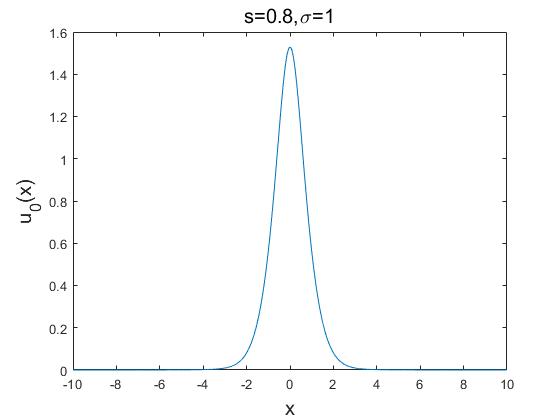}}
  \subfigure[Time evolution of $|\psi(x,t)|$]{
   \label{F1.2} \includegraphics[width=0.4\linewidth,height=3cm]{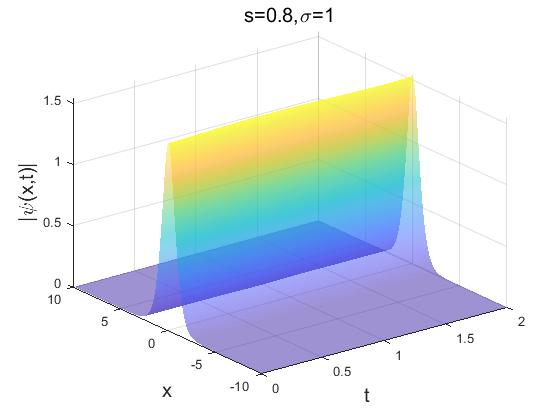}}
\subfigure[Time evolution of $Re(\psi(0,t))$]{
   \label{F1.3} \includegraphics[width=0.4\linewidth,height=3cm]{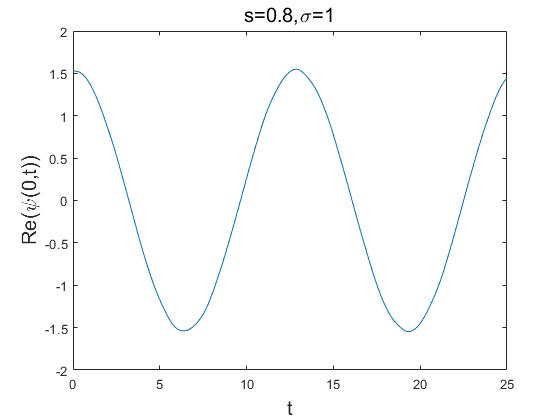}}
\subfigure[Time evolution of $Im(\psi(0,t))$]{
   \label{F1.4} \includegraphics[width=0.4\linewidth,height=3cm]{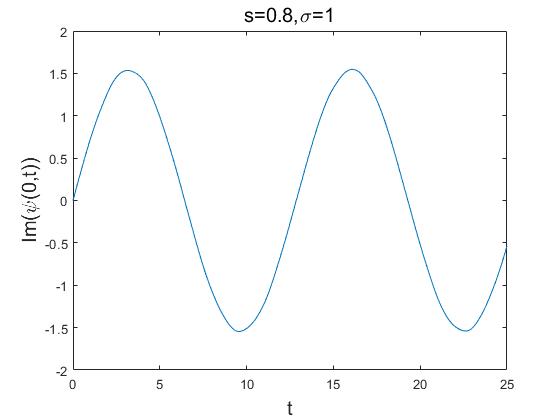}}
   \caption{Ground state solution and time dynamics of standing waves with $s=0.8$, $\sigma=1$, $L=10$ and $J=5000$}
\end{figure}

By Theorem \ref{T1}, we can obtain the existence of ground state solutions with $\sigma<\frac{2s}{N}$. We change $s$ but keep $\sigma=1$ to obtain ground state solutions with different $s$. From figure \ref{F2.1}, we can see when $s$ approaches to $0.5$,  the ground state solution becomes peaked with faster spatial decay. This is a similar result to the case without potential \cite{R4}. We also check $\|\cdot\|_{\Sigma_1(\mathbb{R})}$ of ground state solutions when $s\rightarrow 0.5$ in \eqref{F2.2}, whose growth shows regularity of ground state solution becomes worse.

\begin{figure}[H]
\centering
 \subfigure[Ground state solutions with different $s$]{
\label{F2.1}\includegraphics[width=0.4\linewidth,height=4cm]{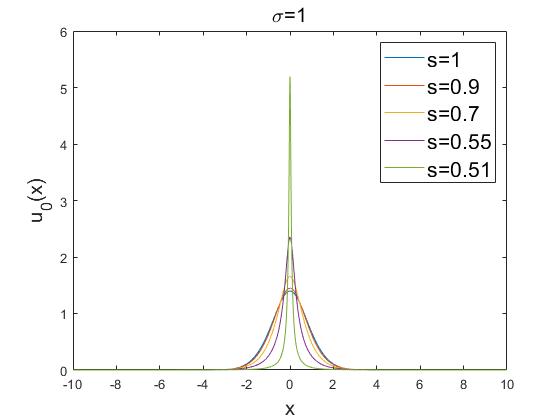}}
 \subfigure[$\|\cdot\|_{\Sigma_1(\mathbb{R})}$ of ground state solutions with different $s$]{
\label{F2.2}\includegraphics[width=0.4\linewidth,height=4Cm]{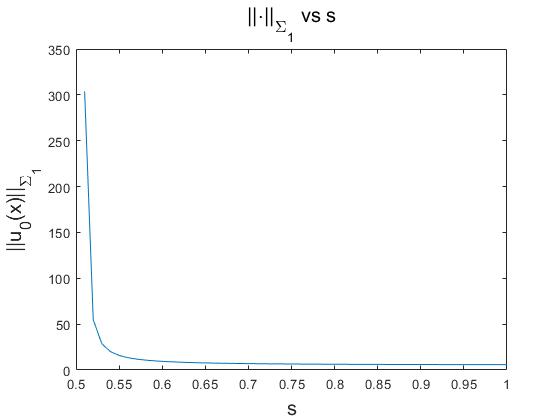}}
\caption{Ground state solutions with $\sigma=1$ and different $s$}
\label{F2}
\end{figure}

From \eqref{F2.1} and \eqref{F2.2}, it seems ground state solutions change continuously with $s$. We use $L^2$ distance between $u^s_0(x)-u^1(x)$ to check and see the convergence of ground state solutions in $L^2$ space with $s\rightarrow 1$. (Figure \ref{F3.1})
\begin{figure}[h!]
\centering
\includegraphics[width=0.7\linewidth,height=5cm]{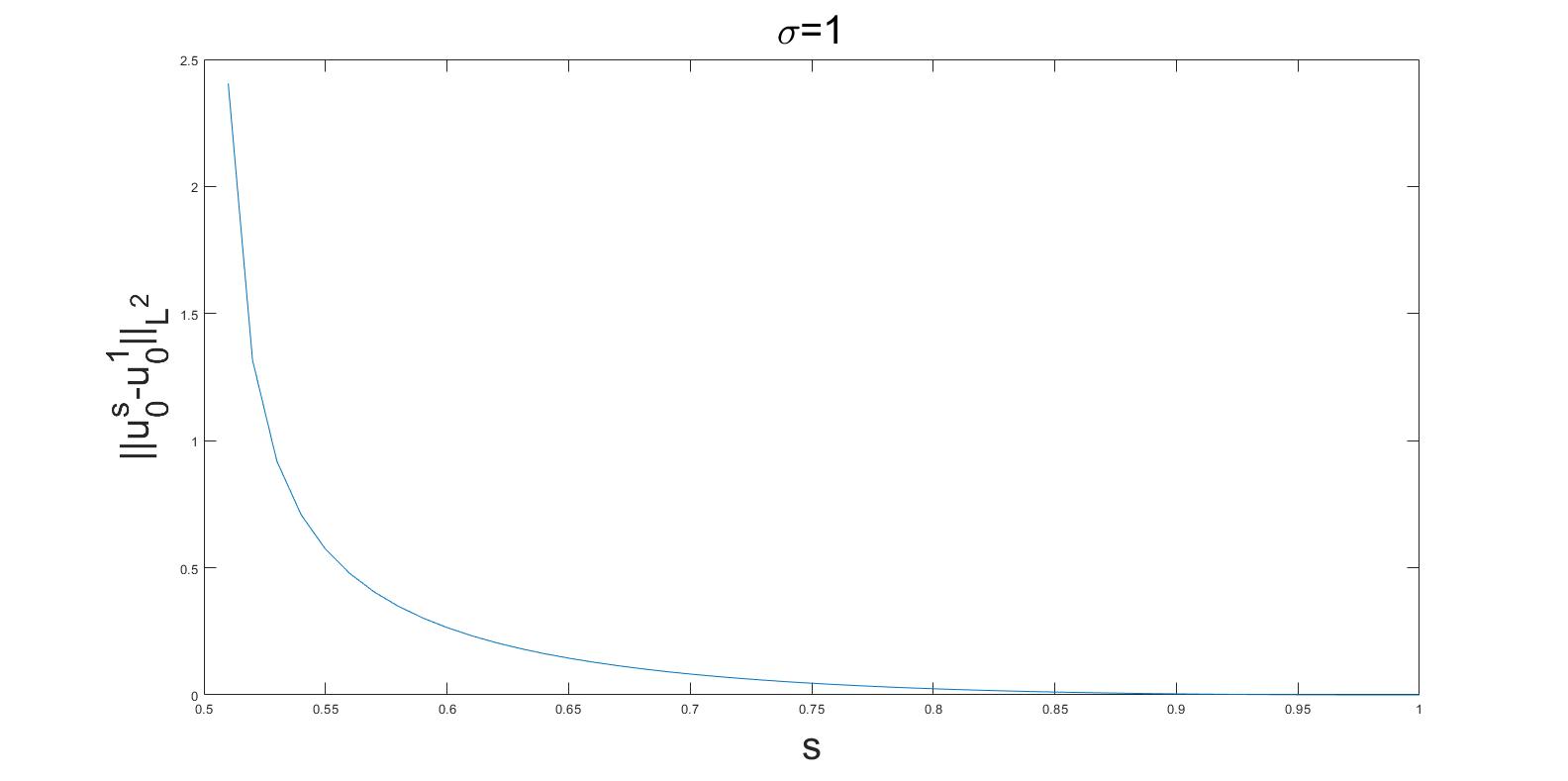}
\caption{$L^2$ distance between ground state solutions of $s<1$ and $s=1$ when $\sigma=1$}
\label{F3.1}
\end{figure}

Then, we test another two things. The first is the relation between the constrained minimal energy in \eqref{E2.4} and $s$. We calculate the discrete energy by
\begin{equation}\label{E7.1}
E(s)=h\Sumb\left[\Suma|\mu_l|^{2s}|\widehat{u}_l|^2+|x_j|^2|u_j|^2-\frac{1}{\sigma+1}|u_j|^{2(\sigma+1)}\right].
\end{equation}
From figure \ref{F11.1}, we find the energy's dependence $(E(s))$ on $s$ is monotonic. When $s$ approaches to $0.5$, the energy will approach to $-\infty$ because of focusing nonlinear term. There are two reasons. First, we keep the $L^2$ norm of $u$ (we test with same mass $c$), but the potential term becomes small since $u$ gathers around $0$. Second, in Lemma \ref{L3.2}, we need $\sigma<\frac{2s}{N}$ to bound $\|u\|_{2\sigma+2}$, whose boundedness becomes worse when $s\rightarrow0.5$. This is different from \cite{R4}, where they didn't use the variation form and keep the $L^2$ norm. 

Second, we test the relationship between mass $c$ and $\lambda_c$, where $c$, $\lambda_c$ are mass and Lagrange multiplier corresponding to the minimizer \eqref{E2.4}-\eqref{E2.7} . By \cite{R3}, in the case $s=1$, there exits $\lambda_0$ such that for any $\sigma$,
\[
\lim_{c\rightarrow 0}\lambda_c=\lambda_0.
\]
We also test this with $s=0.8$. From figure \eqref{F11.2}, we can see for different $\sigma$, when $c\rightarrow 0$, $\lambda_c$ will also converge to a same value $\lambda_0$.
\begin{figure}[H]
\centering
 \subfigure[Evolution of $E(s)$ with different $s$]
 {\includegraphics[width=0.45\linewidth,height=4cm]{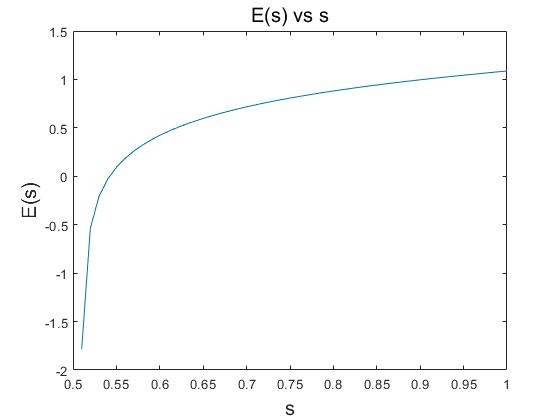}
\label{F11.1}}
\subfigure[Evolution of $\lambda_c$ with different $c$]
{\includegraphics[width=0.45\linewidth,height=4cm]{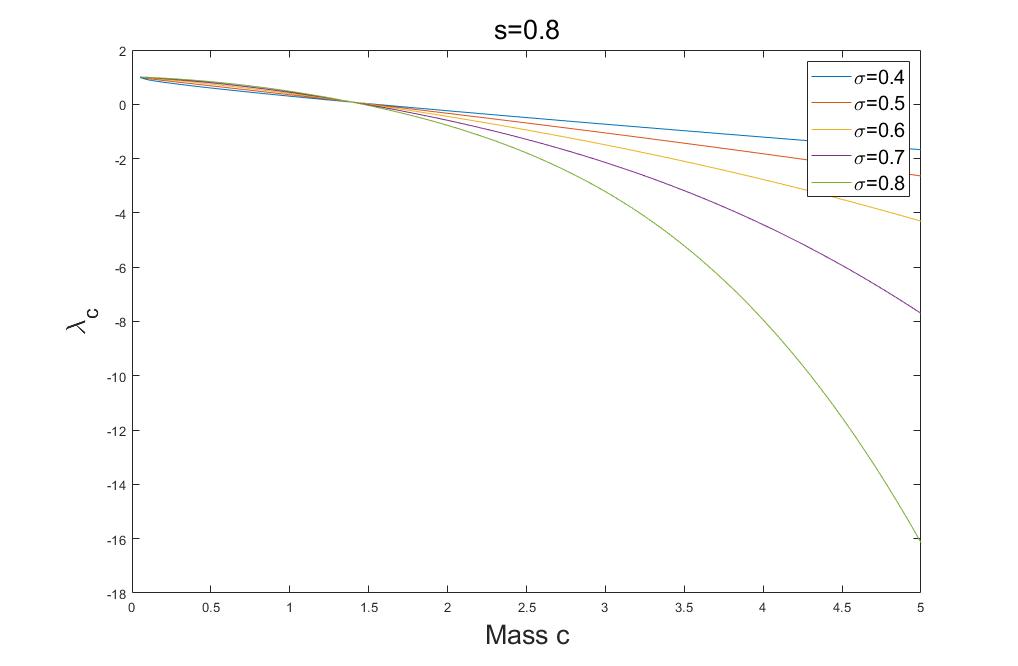}
\label{F11.2}}
\label{F11}
\caption{Energy and $\lambda_c$}
\end{figure}

Up to now, we only consider the case with radial symmetrical potential. However, when potential is not radially symmetric, we can still find standing waves to \eqref{E1.1} using \eqref{E2.4}. We tried the case where potential is $|x|^2+a\sin(2\pi x)$ with $a=1$ and $a=5$. From figure \ref{F10.1} and \ref{F10.2}, we see if we add a nonsymmetrical perturbation to potential, we won't get radially symmetrical ground state solutions.
 \begin{figure}[H]
  \centering
 \subfigure[potential term=$|x|^2+\sin(2\pi x)$]
    {\label{F10.1}\includegraphics[width=0.45\linewidth,height=4cm]{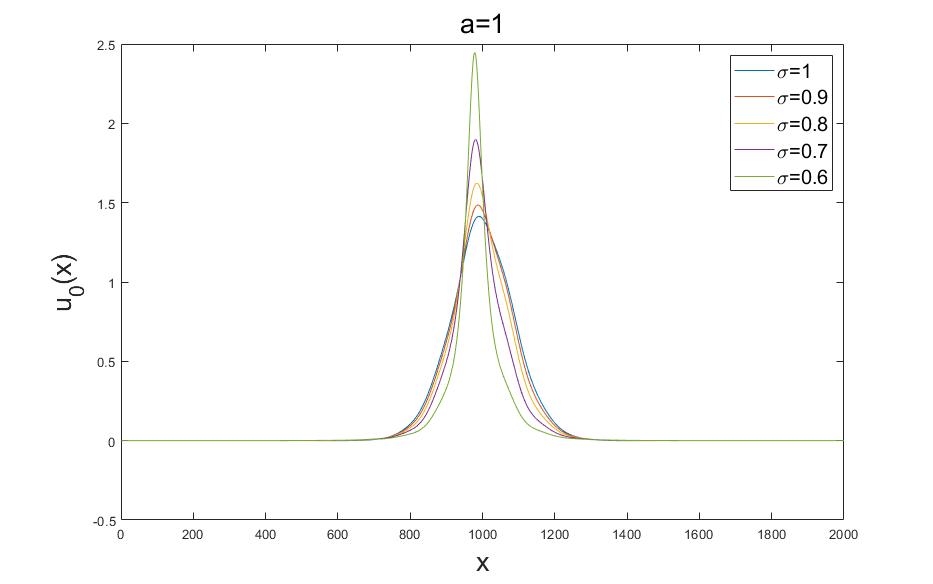}}
  \subfigure[potential term=$|x|^2+5\sin(2\pi x)$]{
   \label{F10.2}\includegraphics[width=0.45\linewidth,height=4cm]{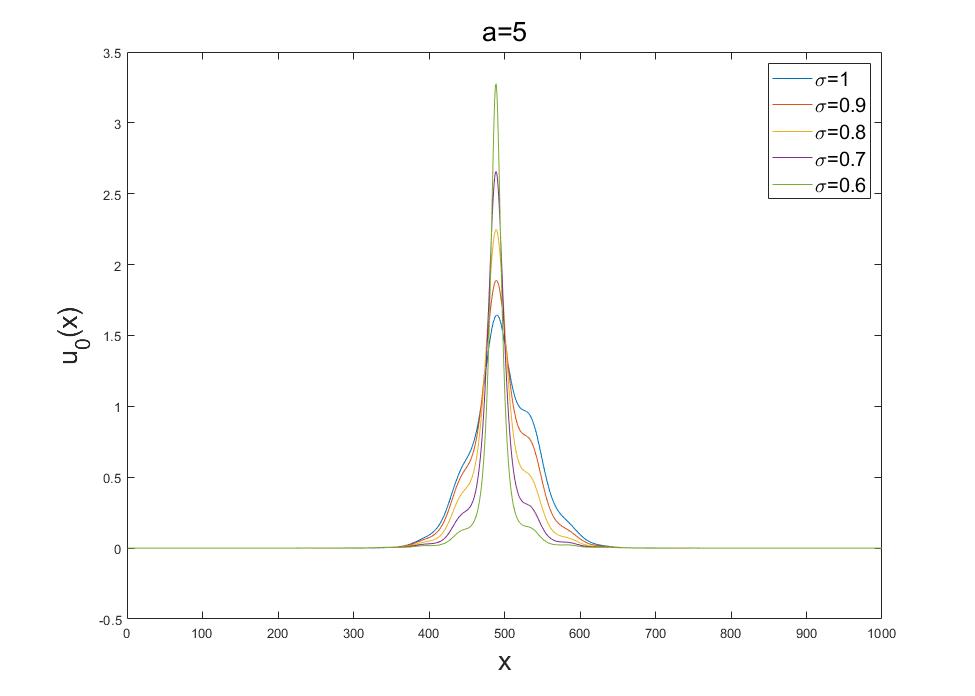}}
   \caption{Ground state solutions with non-symmetric potential}
\end{figure}

\subsection{Dynamics and stability of ground state solution}

First, we consider the case $s>\frac{N\sigma}{2}$, which is covered by our Theorem \ref{T2}. From theoretic results and figure \ref{F1.2}, we can see standing waves preserve $|\psi(x,t)|$ with time $t$. Therefore, we use $|\psi(x,t)|$ to draw graphs and test its stability. We consider the case $s=0.8$ and $\sigma=1$. We first test condition \eqref{E4.19} in Theorem \ref{T2} with initial condition $\psi_0(x)=(1+e)*u_0(x)$, because the scheme is mass preserving, it suffices to test energy preservation, which is showed in figure \ref{Fec1},\ref{Fec2}. Then, we test the stability of solution, where $e$ is a constant number. From figure \ref{F4.1},\ref{F4.2}, we can see when $e=0.05$, the solution almost preserves $|\psi(x,t)|$ as we desired. When $e=0.2$, the solution shows large perturbation but still has periodic behavior, similar to \cite{R3} and \cite{R4}. In figure \ref{FEs}, we compare $\|\cdot\|_{\widetilde{\Sigma}_s}$ distance between ground state solution and perturbed solution.
 \begin{figure}[H]
  \centering
   \subfigure[Energy check when $\psi^s_0(x)=(1.05)*u^s_0(x)$]
    {\label{Fec1}\includegraphics[width=0.45\linewidth,height=4cm]{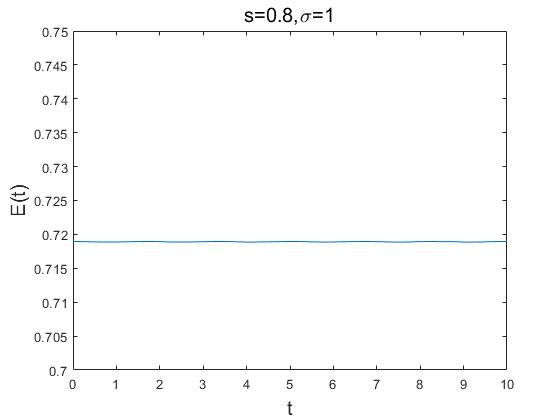}}
  \subfigure[Energy check when $\psi^s_0(x)=(1.2)*u^s_0(x)$]{
   \label{Fec2} \includegraphics[width=0.45\linewidth,height=4cm]{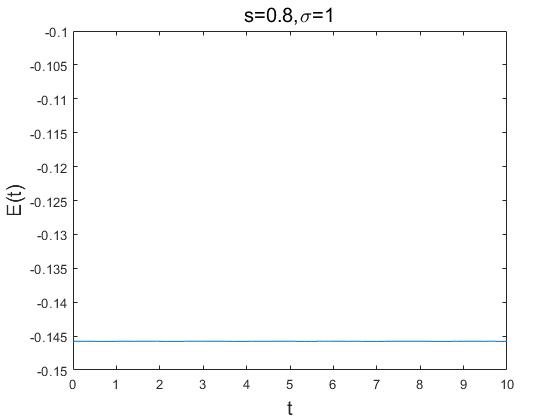}}
 \subfigure[Absolution value of solution when $\psi^s_0(x)=(1.05)*u^s_0(x)$]
    {\label{F4.1}\includegraphics[width=0.45\linewidth,height=4cm]{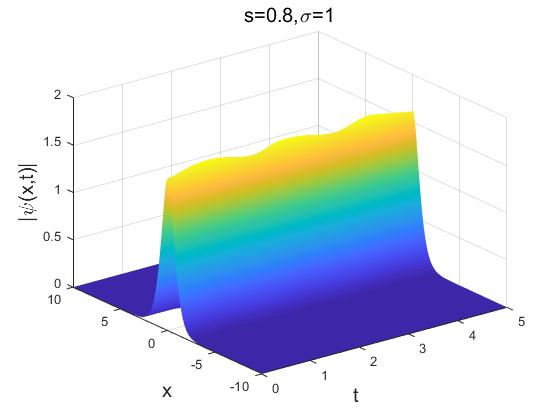}}
  \subfigure[Absolution value of solution when $\psi^s_0(x)=(1.2)*u^s_0(x)$]{
   \label{F4.2} \includegraphics[width=0.45\linewidth,height=4cm]{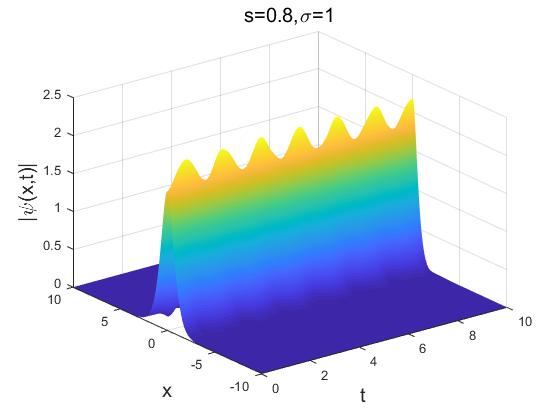}}
   \caption{Energy and stability check with $s=0.8$, $\sigma=1$}
   \label{F4}
\end{figure}
\begin{figure}[h!]
\centering
\includegraphics[width=0.7\linewidth,height=5cm]{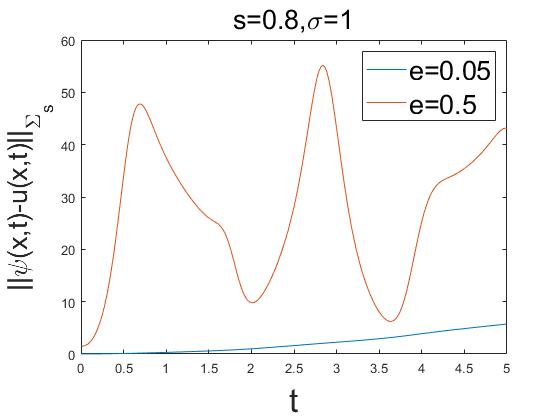}
\caption{$\|\psi^s(x,t)-u^s(x,t)\|_{\widetilde{\Sigma}_s}$ when $s=0.8$ and $\sigma=1$}
\label{FEs}
\end{figure}
We can see from figure \ref{F2.2} and theoretic results that when $s\rightarrow \frac{\sigma N}{2}$, the regularity of ground state solutions becomes worse. This inspires us to investigate its stability relationship with $s$. By Theorem \ref{T2}, Def \ref{D4.1}, the orbital stability means we can find $\omega\in\widetilde{\mathcal{O}}_{c}$ such that $\|\omega-\psi\|_{\widetilde{\Sigma}_{s}(\mathbb{R})}$ is small when we only have small perturbation in initial condition. This definition is hard to measure. Therefore, instead of checking the exact definition of orbital stability, we test classical stability by  comparing distance between perturbed solution and ground state solution using normalized $\|\cdot\|_{\widetilde{\Sigma}_s(\mathbb{R})}$ distance:
\[
D(s,t)=\frac{\|u^s(x,t)-\psi^s(x,t)\|_{\widetilde{\Sigma}_{s}(\mathbb{R})}}{\|u^s(x,t)\|_{\widetilde{\Sigma}_{s}(\mathbb{R})}}.
\]
We test initial condition $\psi_0(x)=0.9*u_0(x)$ with $s=0.8$ and $s=0.6$. From figure \ref{F5},\ref{F6}, as expected, when $s$ is small, its stability seems worse. To be complete, we test $D(s,1)$ with $s$ between $0.51$ and $1$ in figure \ref{F7}. We can see $D(1,s)$ increases when $s$ approaches to $0.5$, which implies worse stability. 

 \begin{figure}[H]
  \centering
 \subfigure[Abosolute value of solution with $\psi^s_0(x)=0.9*u^s_0(x)$ and $s=0.8$]
    {\label{F5.1}\includegraphics[width=0.4\linewidth,height=4cm]{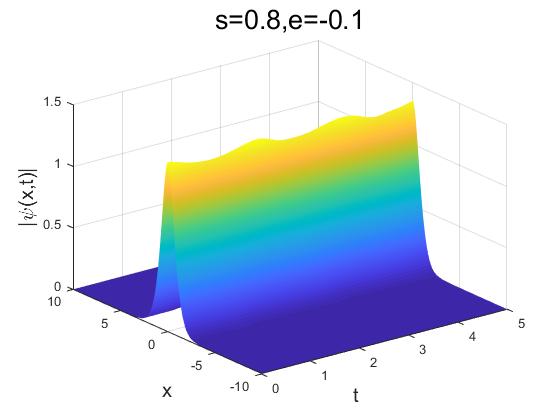}}
  \subfigure[Abosolute value of solution with $\psi^s_0(x)=0.9*u^s_0(x)$ with $s=0.6$]{
   \label{F5.2} \includegraphics[width=0.4\linewidth,height=4cm]{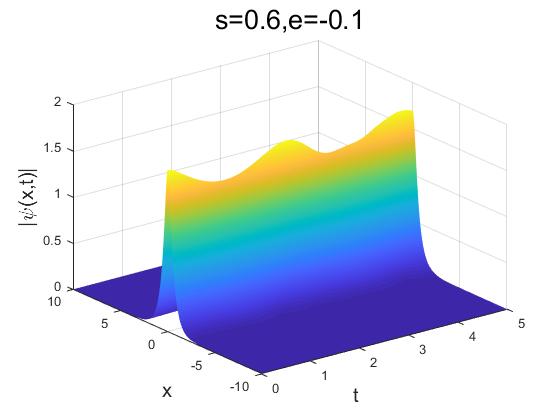}}
   \caption{Abosolute of solution with different $s$}
   \label{F5}
\end{figure}
 \begin{figure}[H]
  \centering
 \subfigure[$D(0.8,t)$ vs $D(0.6,t)$]
    {\label{F6}\includegraphics[width=0.4\linewidth,height=4cm]{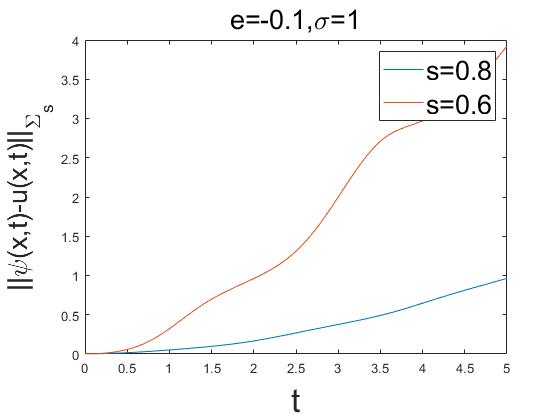}}
  \subfigure[$D(1,t)$ with different $s$]
   {\label{F7} \includegraphics[width=0.4\linewidth,height=4cm]{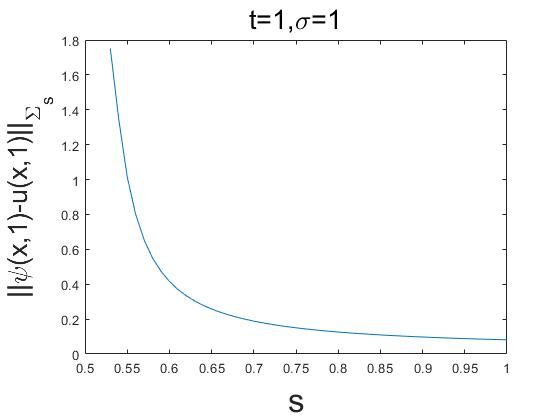}}
   \caption{Stability test with $\psi^s_0(x)=0.9*u^s_0(x)$}
\end{figure}

Second, we try to obtain some numerical result when we touch the critical point $s=\frac{\sigma N}{2}$. In this case, we can't find the ground state solution through \eqref{E2.4} because $I_c=-\infty$. However, as we discussed before, we can find a ground state solution related to another constrained minimization problem \eqref{E6.15}. Here, we try to use the NGF method to find the ground state solution with $s=0.5$, $\sigma=1$. We first tried positive $\omega$, but the projection step dominated the process \eqref{NGF2}. Therefore, we tried $\omega=-0.5$ and find the method does converge to a solution. From figure \ref{F052}-\ref{F056}, we can see $|\psi(x,t)|$ almost preserves with time $t$ with periodical real and imaginary part. We use it to test the finite blow up phenomenon ($\psi_0(x)=2u_0(x)$) appeared in the case without potential \cite{R4}, and this also happens with potential (Figure \ref{F058}). We note here we still can't find a perfect ground state solution, the reason might come from when $s=0.5$, the stability of \eqref{E1.1} is very bad.
\begin{figure}[!h]
\centering
\subfigure[Standing waves when $s=0.5$]{\includegraphics[width=0.4\linewidth,height=3cm]{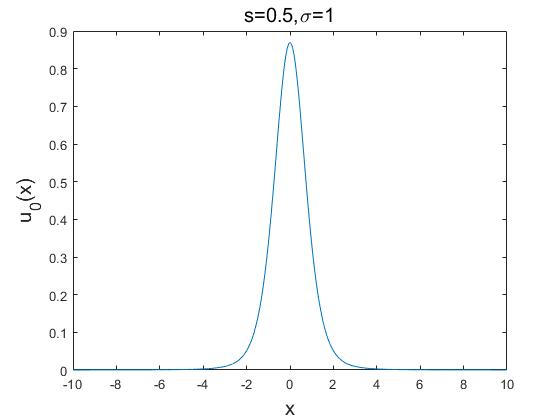}
\label{F051}}
\subfigure[Absolute value of ground state solution w $s=0.5$]{\includegraphics[width=0.4\linewidth,height=3cm]{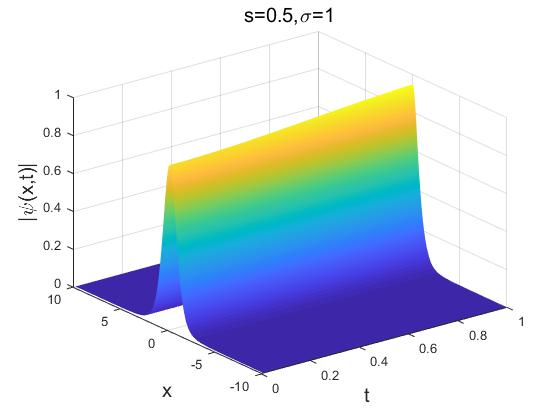}
\label{F052}}
\subfigure[Time evolution of $Re(\psi(0,t))$]{\includegraphics[width=0.4\linewidth,height=3cm]{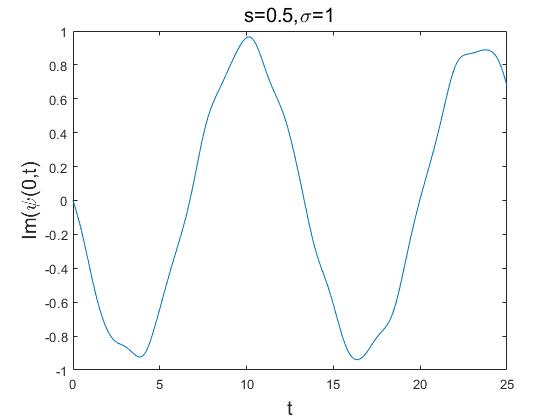}
\label{F055}}
\subfigure[Time evolution of $Im(\psi(0,t))$]{\includegraphics[width=0.4\linewidth,height=3cm]{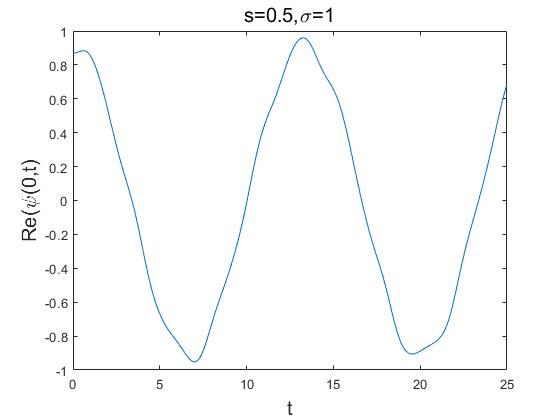}
\label{F056}}
\subfigure[Blow up of solution with $\psi_0(x,t)=2u_0(x)$]{\includegraphics[width=0.4\linewidth,height=3cm]{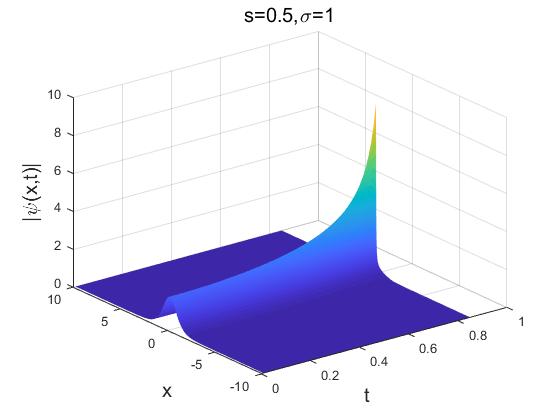}}
\label{F058}
\caption{Standing wave and ground state solution when $s=0.5$, $\delta=1$}
\end{figure}

Finally, we test some simple time dynamics of FNLS, we let $\psi_0(x)=u_0(x)e^{ikx}$, which changes its phase but not absolute value. If $s=0.8$ and $k=1,20$ (Figure \ref{F9.1}, \ref{F9.2}), the maximum point of $|\psi(x,t)|$ will move along $x$ periodically. We also test the $L^\infty$ norm of $\psi(x,t)$ (Figure \ref{F9.3}). We find it decreases first and then approaches to some limits, which is similar to the case without potential \cite{R4}. 

 \begin{figure}[!htbp]
  \centering   
 \subfigure[Abolute value of solution when $s=0.8,\psi_0(x)=u_0(x)e^{ix}$.]
    {\label{F9.1}\includegraphics[width=0.45\linewidth,height=3.3cm]{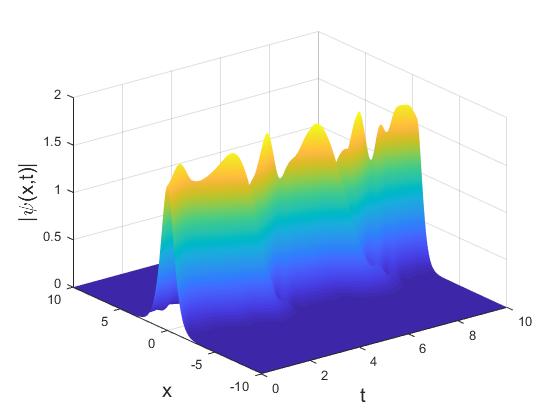}}
  \subfigure[Abolute value of solution when $s=0.8,\psi_0(x)=u_0(x)e^{i20x}$.]{
   \label{F9.2} \includegraphics[width=0.45\linewidth,height=3.3cm]{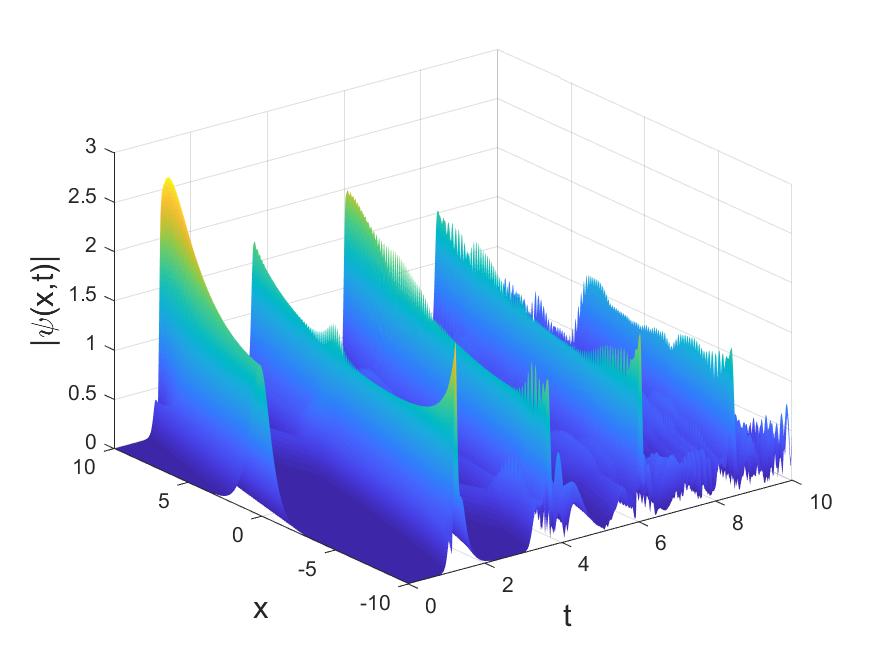}}
   \subfigure[$L^\infty$ evolution of the solution when $s=0.8,\psi_0(x)=u_0(x)e^{i20x}$.]{
   \label{F9.3} \includegraphics[width=0.45\linewidth,height=3.3cm]{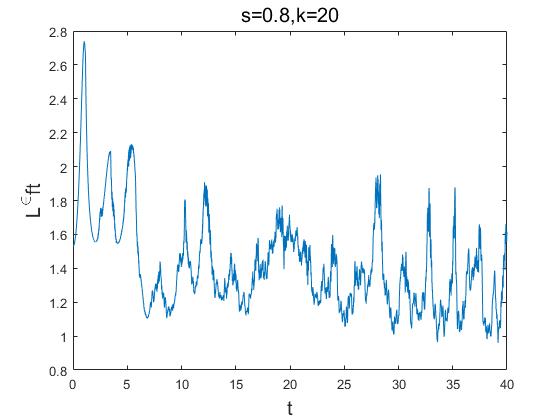}}
   \caption{Dynamics of FNLS}
\end{figure}
\bibliographystyle{siam}
\bibliography{Manuscript}

\begin{thebibliography}{10}

\bibitem{R8}
{\sc F.~J. Almgren and E.~H. Lieb}, {\em Symmetric decreasing rearrangement is
  sometimes continuous}, Journal of the American Mathematical Society, 2
  (1989), pp.~683--773.

\bibitem{R2}
{\sc W.~Bao and Q.~Du}, {\em Computing the ground state solution of
  bose--einstein condensates by a normalized gradient flow}, SIAM J. Sci.
  Comput., 25 (2004), pp.~1674--1697.

\bibitem{haj2}
{\sc R.~Carles and H.~Hajaiej}, {\em Complementary study of the standing wave
  solutions of the gross-pitaevskii equation in dipolar quantum gases},
  Bulletin of the London Mathematical Society, 47 (2014).

\bibitem{R7}
{\sc T.~Cazenave and P.-L. Lions}, {\em Orbital stability of standing waves for
  some nonlinear schr\"odinger equations}, Communications in Mathematical
  Physics, 85 (1982).

\bibitem{cha}
{\sc X.~Chang}, {\em Ground state solutions of asymptotically linear fractional
  schrödinger equation}, Journal of Mathematical Physics, 54 (2013).

\bibitem{chan}
{\sc M.~Cheng}, {\em Bound state for the fractional schrödinger equation with
  unbounded potential}, Journal of Mathematical Physics, 53 (2012).

\bibitem{R1}
{\sc S.~Duo and Y.~Zhang}, {\em Mass-conservative fourier spectral methods for
  solving the fractional nonlinear schrödinger equation}, Computers and
  Mathematics with Applications,  (2016).

\bibitem{fib}
{\sc G.~Fibich and X.-P. Wang}, {\em Stability of solitary waves for nonlinear
  schr\"odinger equations with inhomogeneous nonlinearities}, Physica D:
  Nonlinear Phenomena, 175 (2003), p.~96–108.

\bibitem{R3}
{\sc F.~Hadj~Selem, H.~Hajaiej, P.~Markowich, and S.~Trabelsi}, {\em
  Variational approach to the orbital stability of standing waves of the
  gross-pitaevskii equation}, Milan Journal of Mathematics, 82 (2014).

\bibitem{OB}
{\sc H.~Hajaiej}, {\em On the optimality of the assumptions used to prove the
  existence and symmetry of minimizers of some fractional constrained
  variational problems}, Annales Henri Poincar{\'e}, 14 (2013), pp.~1425--1433.

\bibitem{haj}
{\sc H.~Hajaiej, P.~A. Markowich, and S.~Trabelsi}, {\em Minimizers of a class
  of constrained vectorial variational problems: Part i.}, Milan Journal of
  Mathematics, 82 (2014), pp.~81--98.

\bibitem{HOW}
{\sc H.~Hajaiej, L.~Molinet, T.~Ozawa, and B.~Wang}, {\em Sufficient and
  necessary conditions for the fractional gagliardo-nirenberg inequalities and
  applications to navier-stokes and generalized boson equations}, RIMS
  Kokyuroku Bessatsu,  (2010).

\bibitem{haj1}
{\sc H.~Hajaiej and C.~Stuart}, {\em On the variational approach to the
  stability of standing waves for the nonlinear schrödinger equation},
  Advanced Nonlinear Studies, 4 (2004).

\bibitem{R5}
{\sc H.~Hajaiej and C.~A. Stuart}, {\em Symmetrization inequalities for
  composition operators of carathéodory type}, Proceedings of the London
  Mathematical Society, 87 (2003), p.~396–418.

\bibitem{R4}
{\sc C.~Klein, C.~Sparber, and P.~Markowich}, {\em Numerical study of
  fractional nonlinear schrodinger equations}, Proceedings. Mathematical,
  physical, and engineering sciences / the Royal Society, 470 (2014),
  p.~20140364.

\bibitem{las}
{\sc N.~Laskin}, {\em Fractional quantum mechanics and l\'evy path integrals},
  Physics Letters A, 268 (2000), pp.~298--305.

\bibitem{lask}
{\sc N.~Laskin}, {\em Fractional schr\"odinger equation}, Phys. Rev. E, 66
  (2002), p.~056108.

\bibitem{ros}
{\sc H.~Rose and M.~Weinstein}, {\em On the bound states of the nonlinear
  schrödinger equation with a linear potential}, Physica D: Nonlinear
  Phenomena, 30 (1988), pp.~207--218.

\bibitem{Zh}
{\sc J.~Zhang}, {\em Stability of standing waves for nonlinear schrödinger
  equations with unbounded potentials}, Zeitschrift für angewandte Mathematik
  und Physik, 51 (2000), p.~498.

\bibitem{Zhao}
{\sc F.~Zhao, L.~Zhao, and Y.~Ding}, {\em Existence and multiplicity of
  solutions for a non-periodic schrödinger equation}, Nonlinear
  Analysis-theory Methods and Applications, 69 (2008), pp.~3671--3678.

\end{thebibliography}
\end{document}